\documentclass[11pt]{amsart}%
\pdfoutput=1
\usepackage{mathrsfs}
\usepackage{amsmath}
\usepackage{amsfonts}
\usepackage{amssymb}
\usepackage{amsthm}
\usepackage{graphicx}
\usepackage{subfigure}
\usepackage{color}
\usepackage[margin=1.0in]{geometry}%

\providecommand{\U}[1]{\protect\rule{.1in}{.1in}}
\newtheorem{theorem}{Theorem}

\newtheorem{lemma}[theorem]{Lemma}

\newtheorem{prop}[theorem]{Proposition}
\newtheorem{remark}[theorem]{Remark}
\newtheorem{problem}[theorem]{Problem}
\newtheorem*{condition}{Visibility condition}
\begin{document}
\title[Microlocally accurate solution of the inverse source problem]{Microlocally accurate solution of the inverse source problem of thermoacoustic tomography}
\author{M. Eller \and P. Hoskins \and L. Kunyansky}

\maketitle

\begin{abstract}
We consider the inverse source problem of thermo- and photoacoustic
tomography, with data registered on an open surface partially surrounding the
source of acoustic waves. Under the assumption of constant speed of sound we
develop an explicit non-iterative reconstruction procedure that recovers the
Radon transform of the sought source, up to an infinitely smooth additive error
term. The source then can be found by inverting the Radon transform.

Our analysis is microlocal in nature and does not provide a norm estimate on
the error in the so obtained image. However, numerical simulations show that
this error is quite small in practical terms.
We also present an asymptotically fast implementation of this procedure
for the case when the data are given on a circular arc in 2D.

\end{abstract}


\section{Introduction}

We consider the inverse source problem for the free-space wave equation, with
measurements made on an open bounded surface. Such problems arise in the
thermo- and photoacoustic tomography (TAT/PAT)
\cite{Oraev94,KLFA-MP-95,KRK-MP-99} and in several other coupled-physics
modalities (see, for example, \cite{Wang2004,KK-AET-10,Widlak12}). The forward
problem can be modeled by the Cauchy problem for the standard wave equation in
$\mathbb{R}^{d}$
\begin{equation}
c^{2}(x)\Delta u(t,x)=u_{tt}(t,x),\quad t>0,\quad x\in\mathbb{R}^{d},
\label{E:originalwave}%
\end{equation}
with the initial condition%
\begin{equation}
u(0,x)=f(x),\qquad u_{t}(0,x)=0, \label{E:originalBC}%
\end{equation}
where, in the case of TAT/PAT, $u(t,x)$ represents the acoustic pressure in
the tissues, $c(x)\ $is the speed of sound, and the initial pressure $f(x)$
results from the thermoelastic effect. Function $f$ is assumed to be compactly supported
within a bounded region $\Omega_{0}$ that itself is a subset of a larger
bounded region $\Omega\supset\Omega_{0}$. Ideally, one would take the
measurements $g(t,z)$ defined as
\begin{equation}
g(t,z)\equiv u(t,z),\quad z\in\partial\Omega,\quad t\in(0,\infty),
\label{E:ideal-meas}%
\end{equation}
where $\partial\Omega$ is the boundary $\Omega.$ However, here we consider the
practically important case of measurements restricted to a subset $\Gamma$ of
$\partial\Omega;$ we will denote this data $g^{\mathrm{reduced}}(t,z)$:
\begin{equation}
g^{\mathrm{reduced}}(t,z)\equiv\left\{
\begin{array}
[c]{rr}%
g(t,z)=u(t,z), & z\in\Gamma\subset\partial\Omega,\\
0, & z\in\partial\Omega\backslash\Gamma,
\end{array}
\right.  \quad t\in(0,\infty). \label{E:measurements}%
\end{equation}
The two inverse source problems of interest to us consist of finding $f(x)$ from known
$c(x)$ and either $g(t,z)$ or $g^{\mathrm{reduced}}(t,z).$

The inverse source problem with complete data $g(t,z)$ is very well understood
by now. In particular, the so-called time-reversal procedure (applicable under
the non-trapping condition) is based on solving backwards in time the mixed
initial/boundary value problem for the wave equation in the domain
$(0,\infty)\times\Omega,$ with uniform initial conditions at the limit
$t\rightarrow\infty$ and the boundary values $g(t,z)$
\cite{Finch04,HKN-IP-08,AK,US}. Many other techniques have been developed for
the inverse source problem with complete data, including a variety of explicit
inversion formulas applicable if $c(x)$ is constant (we will not attempt to
overview these results here).

However, in practical situations one can measure data only on a subset
$\Gamma$ of $\partial\Omega,$ and has to work with $g^{\mathrm{reduced}}.$ For
such data very few reconstruction technique are presently known, most of them
iterative in nature. In the present paper under the assumption of the constant
speed of sound, we develop an explicit non-iterative image reconstruction
procedure that recovers $f(x)$ up to an infinitely smooth error term. Our
numerical experiments show that this error is also quite small from the
practical standpoint, although we are able to prove only the smoothness of it.

\begin{figure}[t]
\begin{center}
\includegraphics[scale = 0.8]{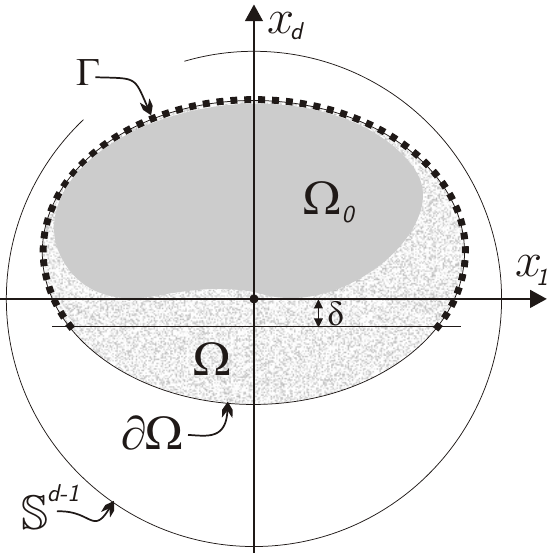}
\end{center}
\caption{Data acquisition geometry}%
\label{F:geometry}%
\end{figure}

We make the following assumptions about the data acquisition scheme. The
region $\Omega$ containing $\Omega_{0}$ is convex and has a smooth boundary
$\partial\Omega.$ Without further loss of generality we assume that $\Omega$
is contained inside the unit sphere $\mathbb{S}^{d-1}$ centered at the origin
(see figure \ref{F:geometry}), and that the constant speed of sound $c$ is
equal to 1. The support $\Omega_{0}$ of $f$ lies within $\Omega$ and above
the plane $x^{d}=0.$ The measurement surface $\Gamma$ is a subset of
$\partial\Omega$ containing all points lying above the plain $x^{d}=-2\delta,$
where $\delta$ is a fixed positive number:%
\[
\Gamma=\{x \ : \ x\in\partial\Omega\text{ and }x^{d}\geq-2\delta\}.
\]
Our goal is to solve

\begin{problem}
(Inverse source problem with reduced data). Find $f(x)$ given data
$g^{\mathrm{reduced}}(t,z)$ defined by equations (\ref{E:originalwave}),
(\ref{E:originalBC}) and (\ref{E:measurements}) $.$
\end{problem}

However, it will be helpful for us to start the discussion with

\begin{problem}
(Inverse source problem with complete data). Find $f(x)$ given data $g(t,z)$
defined by equations (\ref{E:originalwave}), (\ref{E:originalBC}), and
(\ref{E:ideal-meas}).
\end{problem}

The solvability and stability of the inverse problem with reduced data have
been thoroughly studied under various assumptions on the speed of sound
$c(x)\ $(see \cite{US,Finch04,HKN-IP-08}). Importantly, the stability of the
reconstruction is guaranteed under the so-called visibility condition
\cite{KuKuHand}, \ which for the case of constant speed of sound can be
formulated as follows

\begin{condition}
For every point $x\in\Omega_{0},$ every straight line passing through $x$
intersects $\Gamma$ at least once.
\end{condition}

Our data acquisition scheme satisfies this condition.

The most straightforward constructive approach to the problem with the reduced
data is to apply the standard time-reversal technique to $g^{\mathrm{reduced}%
}$ to obtain an approximation to $f.$ It is shown in~\cite{US} that, under the
visibility condition, such approximation is equivalent to applying a
pseudo-differential operator ($\psi$DO) of order zero to $f.$ A Neumann series
approach was used in \cite{US-Num} to invert such an operator; however,
convergence of this approach is not well understood theoretically. Instead of
the time reversal one can apply the operator adjoint to the operator defining
the direct map of to $g^{\mathrm{reduced}}$ \cite{Halt19}. This still leads to
expensive iterations, although they are guaranteed to converge eventually.

Non-iterative approaches to this problem include
\cite{PS1,PS2,Kun-open,Kun-Do}. The numerical algorithm \cite{Kun-open}
produces good results in 2D; however, it requires further theoretical
investigation and it might become expensive for 3D computations. The method
presented in \cite{Kun-Do} yields a theoretically exact reconstruction of $f,$
but only under a certain sub-optimal geometrical condition: the support of
$f$ must be significantly smaller than the region that otherwise satisfies the
visibility condition. Finally, in \cite{PS1,PS2} the source term $f$ is
reconstructed up to an error term given by a $\psi$DO of order $-1$ applied to
$f.$

In the preset paper we first develop a novel explicit procedure that yields
theoretically exact reconstruction of the Radon projections of $f$ if the
complete data $g$ are given. When applied to reduced data, this algorithm
produces the Radon projections accurate up to an infinitely smooth term. This
allows for an explicit reconstruction of $f$ up to a smooth error. In
Sections~\ref{S:wavefront} and ~\ref{S:Radon} we formulate and prove the main
theorems defining our method. An asymptotically fast algorithm for the
particular case of the circular arc acquisition curve in 2D is presented in
Section~\ref{S:algorithm}. In the last section we demonstrate the performance
of this approach in numerical simulations.

\section{Explicit reconstruction procedure}

\subsection{Theoretically exact reconstruction from complete data}

Let us first introduce an explicit reconstruction procedure that is
theoretically exact in the case of complete data. In other words, in this
section we are solving Problem 2.
\subsubsection{Exterior problem}

Unlike the classical approach that consists
in solving backwards in time the wave equation in a domain $(0,\infty
)\times\Omega,$ we first solve an exterior Dirichlet problem for the wave
equation in the domain $Q^{\mathrm{ext}}\equiv(0,\infty)\times(\mathbb{R}%
^{d}\backslash\bar{\Omega}).$ In order words, we find solution $u(t,x)$ of the
initial/boundary value problem%
\begin{equation}%
\begin{split}
u_{tt}(t,x)-\Delta u(t,x)  &  =0,\quad(t,x)\in Q^{\mathrm{ext}}\;,\\
u(0,x)  &  =0,\quad u_{t}(0,x)=0,\quad x\in\mathbb{R}^{d}\backslash\Omega,\\
u(t,z)  &  =g(t,z),\quad z\in\partial\Omega,\quad t\in(0,\infty).
\end{split}
\label{E:fulldataexterior}%
\end{equation}
Notice that the solution of this problem coincides on $Q^{\mathrm{ext}}$ with
the solution of the problem (\ref{E:originalwave}), (\ref{E:originalBC}),
which allows us to use the same notation $u(t,x)$ for both functions.

\subsubsection{Radon transform}
The Radon transform is the main tool we use. To make this paper self-contained we
collect in this section some basic facts about it.
Consider a continuous function $q(x)$ compactly
supported within some bounded region $\Omega^{\text{\textrm{large}}}%
\subset\mathbb{R}^{d}$ (within the ball $B^{d}(0,r)$ of radius $r$ centered at
the origin) and extended by zero to the rest of $\mathbb{R}^{d}$. For a
direction $\omega\in\mathbb{S}^{d-1},$ and $p\in\mathbb{R}$, the values
$h(\omega,p)$ of the Radon transform $\mathcal{R}q(\omega,p)$ are given by the
formula
\begin{equation}
h(\omega,p)=\left[  \mathcal{R}q\right]  (\omega,p)\equiv\int%
\limits_{\mathbb{R}^{d}}q(x)\delta(x\cdot\omega-p)dx,\quad p\in\mathbb{R},
\label{E:Radontr}%
\end{equation}
where $\delta(\cdot)$ is the Dirac's delta function (\cite{natterer-86}). The
Radon transform is thus defined on the cylinder $Z=\mathbb{S}^{d-1}%
\times\mathbb{R}$. The obvious but important property of $\mathcal{R}$ is
symmetry%
\begin{equation}
\left[  \mathcal{R}q\right]  (\omega,p)=\left[  \mathcal{R}q\right]
(-\omega,-p). \label{E:symmetry}%
\end{equation}

Let us introduce Sobolev norms on $\mathbb{R}^{d}$ and $Z$:%
\begin{align*}
||q||_{H^{s}(\mathbb{R}^{d})}^{2}  &  =\int\limits_{\mathbb{R}^{d}}%
(1+|\xi|^{2})|\hat{q}(\xi)|^{2}d\xi,\\
||h||_{H^{s}(Z)}^{2}  &  =\int\limits_{\mathbb{S}^{d-1}}\int%
\limits_{\mathbb{R}}(1+|\sigma|^{2})|\hat{h}(\omega,\sigma)|^{2}d\sigma d\omega
,
\end{align*}
where $\hat{q}(\xi)$ is the Fourier transform of $q(x)$ and $\hat{h}%
(\omega,\sigma)$ is the one-dimensional Fourier transform of $h(\omega,p)$ in
the second variable. Then there exist constants $c(s,d)$ and $C(s,d)$ (see
\cite{natterer-86}, Chapter 2, Section 5) such that%
\[
c(s,d)||q||_{H^{s}(\mathbb{R}^{d})}\leq||\mathcal{R}q||_{H^{s+(d-1)/2}(Z)}\leq
C(s,d)||q||_{H^{s}(\mathbb{R}^{d})}^{2}.
\]
Using these estimates, one extends the Radon transform to compactly supported
distributions on $\mathbb{R}^{d}.$

\subsubsection{Reconstruction procedure}
We will denote by
$F(\omega,p)$ the Radon transform of $f$%
\begin{equation}
F(\omega,p)\equiv\left[  \mathcal{R}f\right]  (\omega,p).
\label{E:Radon_projections}%
\end{equation}
Recall that $u(t,x)$ is the solution of (\ref{E:originalwave}),
(\ref{E:originalBC}) with the initial condition $f(x).$
Let us define the Radon projections $\left[  \mathcal{R}u\right]
(t,\omega,p)$ of $u(t,x)$ for each fixed $t\geq0$ by (\ref{E:Radontr}). It is
well known that, for a fixed $\omega,$ such projections solve the 1D wave
equation:%
\begin{equation}
\frac{\partial^{2}}{\partial t^{2}}\left[  \mathcal{R}u\right]  (t,\omega
,p)=\frac{\partial^{2}}{\partial p^{2}}\left[  \mathcal{R}u\right]
(t,\omega,p). \label{E:1Dwave}%
\end{equation}
The above equation is understood either in the classical sense or in the sense
of distributions, depending on the smoothness of $u$. Further, by taking into accout~(\ref{E:originalBC})
one can confirm that
\begin{equation}
\left[  \mathcal{R}u\right]  (0,\omega,p)=F(\omega,p),\quad\frac{\partial
}{\partial t}\left[  \mathcal{R}u\right]  (0,\omega,p)=0. \label{E:1Dinitial}%
\end{equation}
Then, the solution of the Cauchy problem (\ref{E:1Dwave}), (\ref{E:1Dinitial}) is
given by d'Alembert's formula,
\[
\lbrack\mathcal{R}u](t,\omega,p)=\frac{1}{2}\left[  F(\omega,p+t)+F(\omega
,p-t)\right]  .
\]
Since $f(x)$ is compactly supported in $B(0,1),$ its Radon projections
$F(\omega,p)$ are supported on $(-1,1)$ in $p$. Therefore, for $t=2,$
\[
\lbrack\mathcal{R}u](2,\omega,p)=\frac{1}{2}F(\omega,p-2),\quad p\in
\lbrack1,3],
\]
or
\[
F(\omega,p)=2[\mathcal{R}u](2,\omega,p+2),\quad p\in\lbrack-1,1].
\]
Thus, once $u(2,x)$ has been computed, one easily reconstructs the Radon
projections $F$ of $f$ from~$\mathcal{R}u.$

The methods for recovering a function from its Radon projections are
well-known \cite{natterer-86}. Consider the Hilbert transform $\mathcal{H}$
whose action a function of a 1D variable $h(t)$ is given by the formula%
\[
\left[  \mathcal{H}h\right]  (s)=\int\limits_{\mathbb{R}}\frac{h(t)}{s-t}dt,
\]
where the integral is understood in the principal value sense. The adjoint
$\mathcal{R}^{\#}$ of the Radon transform (\ref{E:Radontr}) is defined by its
action on a function $h(\omega,p)$ supported on $Z$ as follows%
\[
\lbrack\mathcal{R}^{\#}h](x)=\int\limits_{\mathbb{S}^{d-1}}h(\omega
,\omega\cdot x)d\omega.
\]
The inverse Radon transform $\mathcal{R}^{-1}$ is the left inverse of
$\mathcal{R}$, and is in general not unique. For the purposes of the present
paper we define $\mathcal{R}^{-1}$ as a composition of the differentiation in
$p,$ Hilbert transform $\mathcal{H}$ in the second variable ($p$), and the
adjoint Radon transform $\mathcal{R}^{\#}$ as follows (\cite{natterer-86},
Chapter 2, Section 2):
\begin{equation}
(\mathcal{R}^{-1}F)(x)\equiv\frac{1}{2(2\pi)^{d-1}}\cdot\left\{
\begin{array}
[c]{rr}%
(-1)^{\frac{d-2}{2}}\left[  \mathcal{R}^{\#}\mathcal{H}\frac{\partial^{d-1}%
}{\partial p^{d-1}}F\right]  (x), & d\text{ is even,}\\
(-1)^{\frac{d-1}{2}}\left[  \mathcal{R}^{\#}\frac{\partial^{d-1}}{\partial
p^{d-1}}F\right]  (x), & d\text{ is odd.}%
\end{array}
\right.  \label{E:Inverse_Radon}%
\end{equation}
If $F(\omega,p)=\left[  \mathcal{R}f\right]  (\omega,p),$ then
$f(x)=[\mathcal{R}^{-1}F](x).$ From the numerical standpoint, the problem of
reconstructing a function from its Radon projections has been studied
extensively, and efficient and accurate algorithms for implementing
(\ref{E:Inverse_Radon}) are well known by now \cite{natterer-86}.

It should be noted that the image reconstruction based on the Radon transform
of the solution of the wave equation has been considered recently in
\cite{HaltFull19}. In that work, a non-traditional data acquisition scheme
yields the values of the whole acoustic field $u(t,x)$, instead of boundary
values $g(t,z)$. The problem we are solving here is different, but our
technique bears some resemblance of the methods of~\cite{HaltFull19}.

From the numerical standpoint it may be preferable to compute the derivative
of the Radon projections
\[
\frac{\partial}{\partial p}F(\omega,p)=\frac{\partial}{\partial p}%
2[\mathcal{R}u](2,\omega,p+2)=-2\left[  \mathcal{R}\frac{\partial u}{\partial
t}\right]  (2,\omega,p+2).
\]
Formulas (\ref{E:Inverse_Radon}) can be easily modified to process
$\frac{\partial}{\partial p}F$ instead of $F,$ since the Hilbert transform
commutes with differentiation. However, for simplicity of the presentation we
will concentrate on reconstruction of $F.$

%Numerical efficiency of reconstructing $f(x)$ from the complete data
%$g(t,z)$ (see Problem 2) using our approach depends heavily on the efficiency
%of the algorithm utilized to solve initial/boundary value problem
%(\ref{E:fulldataexterior}). This efficiency is likely to be inferior to the
%traditional reconstruction techniques, with a notable exception of the case of
%circular or spherical acquisition surfaces. In the latter case one can use the
%separation of variables techniques and obtain very efficient asymptotically
%fast algorithms (see Section~\ref{S:algorithm} and \cite{Kun-cyl}).
%The main goal of this paper, however, is to solve Problem 1 with reduced data;
%for this problem our approach offers some distinctive advantages.

\subsection{Reconstruction from reduced data}

\label{S:reduced-reco}

In order to deal with Problem 1, we first solve a modified initial/boundary
problem in $Q^{\mathrm{ext}}$ as follows. For a fixed $\delta>0$ (see Figure
1) we introduce a function $\psi\in C^{\infty}({\mathbb{R}})$ such that
$0\leq\psi\leq1$, $\psi=1$ for $\tau\geq-\delta$, and $\psi=0$ for $\tau
\leq-2\delta$. We define the function $h\in H^{s}((0,\infty)\times
\partial\Omega)$ by setting
\[
h(t,x)\equiv\psi(x_{3})g^{\mathrm{reduced}}(t,x)=\psi(x_{3})g(t,x).
\]

Consider now the exterior initial/boundary problem with homogeneous initial
data and Dirichlet data $h$ for the wave equation:
\begin{equation}%
\begin{split}
w_{tt}(t.x)-\Delta w(t,x)  &  =0,\quad(t,x)\in Q^{\mathrm{ext}},\\
w(0,x)  &  =0,\quad w_{t}(0,x)=0,\quad x\in R^{d}\backslash\bar{\Omega},\\
w(t,z)  &  =h(t,z),\quad(t,z)\in(0,\infty)\times\partial\Omega.
\end{split}
\label{E:Eller2}%
\end{equation}
For $h\in H^{s}((0,\infty)\times\partial\Omega)$, $s \in\mathbb{R}$,
there exists a unique solution $w\in
H^s_{\mathrm{loc}}(Q^{\mathrm{ext}})$. This result can be proved using the
techniques developed in \cite[Section 24.1]{hormander85}. The local
existence theorem (Lemma 24.1.6) is restricted to non-negative $s$ only
because of a right-hand side term which vanishes in our case. Furthermore,
if $h \in L^2((0,\infty)\times \partial \Omega)$, then the solution of our
initial-boundary value problem is a continuous in time,
that is $w\in C([0,\infty),L^2(\mathbb{R}^d\setminus \Omega))$ \cite{llt86}.
%\marginpar{???? {\color{blue} I guess you could argue that there is a continuous representative in some equivalence class. But I would refrain from calling it a distribution here. }}

Solution
$w(t,x)$ serves as a crude approximation of the exact solution $u(t,x)$
(compare this problem to problem (\ref{E:fulldataexterior})).
Let us consider $w(t,x)$ at $t=2$. Due to the
finite (unit) speed of propagation of waves, $w(2,x)$ is finitely supported
within the ball of radius $3$ centered at the origin. Then, the (exterior)
Radon transform $[\mathcal{R}w](2,\omega,p)$ of $w(2,x)$ can be defined by
(\ref{E:Radontr}) for $1\leq|p|\leq3.$ We only use a subset of these data and
define $\tilde{F}(\omega,p)$ as follows:%
\begin{equation}
\tilde{F}(\omega,p)\equiv\lbrack\mathcal{R}w](2,\omega,p+2),\quad p\in
\lbrack-1,1],\quad\omega\in\mathbb{S}^{d-1}. \label{E:Ftilde}%
\end{equation}
When considered on the whole sphere $\mathbb{S}^{d-1}$ (in $\omega)$ the function
$\tilde{F}(\omega,p)$ is a rather poor approximation to $F(\omega,p).$
However, if one considers only the subset $\mathbb{S}^{d-1}_{+}\equiv\{\omega
\in\mathbb{S}^{d-1}\;:\;w_{d}>-\delta/2\},$ one finds that $\tilde{F}(\omega,p)$
is equal to $F(\omega,p)$ modulo infinitely smooth function of $(\omega,p).$
This fact is proven in the Theorem~\ref{T:main} in~Section~\ref{S:Radon} of this paper.

Motivated by the symmetry relation (\ref{E:symmetry}), we define $G(\omega,p)$
as follows%
\begin{equation}
G(\omega,p)=\left\{
\begin{array}
[c]{rr}%
\tilde{F}(\omega,p), & \omega_{d}\geq0,p\in\lbrack-1,1],\\
\tilde{F}(-\omega,-p), & \omega_{d}<0,p\in\lbrack-1,1].
\end{array}
\right.  \label{E:G}%
\end{equation}
We use $G(\omega,p)$ as an approximation for $F(\omega,p)$ and apply the
inverse Radon transform $\mathcal{R}^{-1}$ (defined by (\ref{E:Inverse_Radon}%
)) to $G(\omega,p)$ to obtain an approximation $\tilde{f}(x)$ of $f(x)$:%

\[
\tilde{f}(x)=\left[  \mathcal{R}^{-1}G\right]  (x).
\]
One can show that $\tilde{f}(x)=f(x)+\kappa(x),$ where $\kappa(x)$ is a
$C^{\infty}(B(0,1)).$ This is the main theoretical result of the present
paper, given by Proposition~\ref{T:final} in Section~\ref{S:Radon}. Although
our analysis is microlocal in nature, and does not yield a norm estimate on
the error $\kappa(x),$ our numerical simulations show that it is quite small
in practical terms.

Our general reconstruction procedure consists of the following steps:

\begin{enumerate}
\item Solve the initial/boundary value problem~(\ref{E:Eller2}) on time
interval $[0,2]$;

\item Compute approximate projections $\tilde{F}$ using~(\ref{E:Ftilde}), and
find $G$ using~(\ref{E:G});

\item Compute $\mathcal{R}^{-1}G$ using~(\ref{E:Inverse_Radon}), to obtain $\tilde{f}$
.
\end{enumerate}

Theoretical foundations of our technique are presented in
Sections~\ref{S:wavefront} and~\ref{S:Radon} below.

Computational efficiency of our scheme depends on particular algorithms used
on each of the steps. On the first step, known methods for solving the exterior initial/boundary
value problems can be used, including, in particular, finite difference schemes~\cite{strikwerda2004} and boundary
layer techniques~\cite{sayas2016}. The second step is
straightforward. The last step, inversion of the Radon transform, is
well-studied by now, and efficient implementations of the
filtration/backprojection algorithm~(\ref{E:Inverse_Radon}) are known.

However, for the geometries that allow for separation of variables in the wave equation, it may be
possible to have a faster implementations of steps 2 and 3. In
Section~\ref{S:algorithm} we present a very fast algorithm applicable when the
acquisition surface is a circular arc (in 2D). Finally, the last section of
the paper illustrates our findings by the results of numerical simulations.

\section{The wave front set of the Dirichlet trace of a wave on a sphere}\label{S:wavefront}

Let us analyze the propagation of singularities in the solution $u(t,x)$ of
the problem (\ref{E:originalwave}), (\ref{E:originalBC}), with $c(x)\equiv1$
and $\mathrm{supp}\,f=\Omega_{0}.$ The solution for this Cauchy problem is
given by the formula
\begin{equation}%
\begin{split}
u(t,x)  &  =\frac{1}{(2\pi)^{3}}\int_{\Omega_{0}}\int_{{\mathbb{R}}^{3}%
}e^{\mathrm{i}[(x-y)\cdot\xi+t|\xi|]}f(y)\,d\xi dy+\frac{1}{(2\pi)^{3}}%
\int_{\Omega_{0}}\int_{{\mathbb{R}}^{3}}e^{\mathrm{i}[(x-y)\cdot\xi-t|\xi
|]}f(y)\,d\xi dy\\
&  =\frac{1}{(2\pi)^{3}}\int_{\Omega_{0}}\int_{{\mathbb{R}}^{3}}[e^{t|\xi
|}+e^{-\mathrm{i}t|\xi|}]e^{\mathrm{i}(x-y)\cdot\xi}f(y)\,d\xi dy.
\end{split}
\label{E:Eller1}%
\end{equation}
Consider the mapping $\mathcal{A}:f\rightarrow g$ where
$g=u\big|_{{\mathbb{R}}_{+}\times\partial\Omega}$. This operator is a sum
of two Fourier integral operators. Since
$\Omega_{0}$ is a compact subset of $B(0,1)$, the operator $\mathcal{A}$ has a
non-degenerate phase function and the mapping $f\rightarrow g$ is continuous
from $H_{0}^{s}(\Omega_{0})$ into $H^{s}(\mathbb{R}_{+}\times\partial\Omega)$
for $s\in{\mathbb{R}}$, see for example \cite{US}. Hence, for $f\in H^{s}%
_{0}(\Omega_{0})$ we obtain $h\in H^{s}((0,\infty)\times\partial\Omega)$,
where $h$ is the function introduced at the beginning of the previous subsection.

At first we will study the wave front set of the Dirichlet trace $g$. In what
follows we abbreviate ${\mathbb{R}}_{+}\equiv(0,\infty)$. Since $g$ is defined
on ${\mathbb{R}}_{+}\times\partial\Omega$ its wave front set is a subset of
$T^{\ast}({\mathbb{R}}_{+}\times\partial\Omega)$, the cotangent bundle of the
surface $\mathbb{R}_{+}\times\partial\Omega$. The first result is of a more
general nature.

\begin{lemma}
\label{L:trace} Let $p\in H^{s}({\mathbb{R}}^{d})$ and let $S$ be a smooth
hypersurface in ${\mathbb{R}}^{d}$. Suppose that $\mathrm{WF}(p)\cap
N(S)=\emptyset$ where $N(S)$ is the conormal bundle of $S$. Then, the trace
operator $p\rightarrow q=p\Big|_{S}$ extends to a continuous operator from
$H^{s}({\mathbb{R}}^{d})$ to $H^{s-1/2}(S)$. Furthermore,
\[
\mathrm{WF}(q)\subset\{(x,\xi-(\xi\cdot\nu_{x})\nu_{x})\in T^{\ast}%
S\setminus0\;:\;(x,\xi)\in\mathrm{WF}(p)\}\;.
\]

\end{lemma}

Here $\nu_{x}$ is a unit normal vector at $x\in S$ and $\xi-(\xi\cdot\nu
_{x})\nu_{x}$ is the orthogonal projection of $\xi$ onto the cotangent space
of $S$ at $x$.

\begin{proof}
The first statement of the lemma is just the trace theorem in Sobolev spaces
provided $s>1/2$. The condition on the wave front set allows us to extend the
trace theorem to all $s\in{\mathbb{R}}$. The concept of the wave front set is
independent of the choice of coordinates. Locally, we may always assume that
$S=\{x_{d}=0\}$. In this case we will write $x^{\prime}=(x_{1},...,x_{d-1})$
and correspondingly, we have $\xi=(\xi^{\prime},\xi_{d})$. In local
coordinates, the second statement of the lemma is
\[
\mathrm{WF}(q)=\{(x^{\prime},\xi^{\prime}) \in T^{\ast}{\mathbb{R}}%
^{d-1}\setminus0\;:\;(x^{\prime},0;\xi^{\prime},\xi_{d})\in\mathrm{WF}%
(p)\}\;,
\]
and the conormal bundle is
\[
N(S)=\{(x^{\prime},0;0,\eta_{d})\in T^{\ast}{\mathbb{R}}^{d}\setminus 0\;:\;(x^{\prime
},0)\in S\}\;.
\]
Fix $x^{\prime}\in S$. Since $\mathrm{WF}(p)\cap N(S)=\emptyset$ we know that
there exists a $\varphi\in C_{0}^{\infty}({\mathbb{R}}^{d})$ satisfying
$\varphi(x^{\prime},0)=1$, such that for all $N\in{\mathbb{N}}$ and $\xi$ in a
conic neighborhoods of $e_{d}$ and $-e_{d}$, we have
\[
|\widehat{\varphi p}(\xi)|\leq C_{N}(1+|\xi|)^{-N}\;.
\]
Furthermore, away from this conic neighborhood, we have $|\xi_{d}|\lesssim
|\xi^{\prime}|$. For $s\leq1/2$, we have $\varphi p\in H_{(1,s-1)}%
({\mathbb{R}}^{d})$. For a definition of these Sobolev spaces we refer to
Definition B.1.10 \cite{hormander85}. This local result can be made global and
yields $p\in H_{(1,s-1)}({\mathbb{R}}^{d})$. Using Theorem B.1.11
\cite{hormander85}, we infer that $q\in H^{s-1/2}(S)$. Recall that for
$p\in\mathcal{S}({\mathbb{R}}^{d})$ we have
\[
q(x^{\prime})=p(x^{\prime},0)=\frac{1}{(2\pi)^{d}}\int_{{\mathbb{R}}^{d}%
}e^{\mathrm{i}x^{\prime}\cdot\xi^{\prime}}\hat{p}(\xi)\,d\xi\;,
\]
and hence,
\begin{equation}
\hat{q}(\xi^{\prime})=\frac{1}{2\pi}\int_{\mathbb{R}}\hat{p}(\xi^{\prime}%
,\xi_{d})\,d\xi_{d}\;. \label{E:Eller4}%
\end{equation}
By continuity this last formula extends to Fourier transforms which are
integrable in $\xi_{d}$ with values in $\mathcal{S}^{\prime}({\mathbb{R}}%
^{d})$. Now suppose that the points $(x^{\prime},0;\xi^{\prime},\xi_{d}%
)\notin\mathrm{WF}(p)$ for all $\xi_{d}\in{\mathbb{R}}$. By the definition of
the wave front set, there exists a function $\varphi\in C_{0}^{\infty
}({\mathbb{R}}^{d})$ such that $\varphi(x^{\prime},0)=1$ and a conic
neighborhood of $U(\xi^{\prime})$ in ${\mathbb{R}}^{d-1}$ such that for all
$N\in{\mathbb{N}}$
\[
(\widehat{\varphi p})(\eta^{\prime},\xi_{d})\leq C_{N}(1+|\eta^{\prime}|^{2}
+|\xi_{d}|^{2})^{-N}\quad\mbox{ for all }\eta^{\prime}\in U(\xi^{\prime
})\mbox{ and }\xi_{d}\in{\mathbb{R}}\;.
\]
Now define $\psi\in C_{0}^{\infty}(S)$ by $\psi(y^{\prime})=\varphi(y^{\prime
},0)$. Then we know that for all $\eta^{\prime}\in U(\xi^{\prime})$
\[
|(\widehat{\psi q})(\eta^{\prime})|\leq\frac{1}{2\pi}\int_{\mathbb{R}%
}|\widehat{\varphi p}(\eta^{\prime},\xi_{d})|\,d\xi_{d}\leq\frac{C_{N}}{2\pi
}\int_{\mathbb{R}}(1+|\eta^{\prime}|^{2}+\xi_{d}^{2})^{-N}\,d\xi_{d}\leq
C_{N}^{\prime}(1+|\xi^{\prime}|^{2})^{-N+1/2}\;,
\]
where $C_{N}^{\prime}=C_{N}\int_{\mathbb{R}}(1+s^{2})^{-N}\,ds/(2\pi)$. Hence,
we conclude that $(x^{\prime},\xi^{\prime})\notin\mathrm{WF}(q)$ and conclude
that
\[
\mathrm{WF}(q)\subset\{(x^{\prime},\xi^{\prime}) \in T^{\ast}{\mathbb{R}%
}^{d-1}\setminus0\;:\;(x^{\prime},0;\xi^{\prime},\xi_{d})\in\mathrm{WF}%
(p)\}\;.
\]

\end{proof}

The wave front set of the solution $u$ to (\ref{E:Eller1}) can be described
very precisely. Over every point $(x,\eta)\in\mathrm{WF}(f)$ there are two
null bicharacteristics, that is lines in $\gamma(s)=(t(s),x(s);\tau(s),\eta(s))\subset T^{\ast}%
{\mathbb{R}}^{d+1}$, $s\in{\mathbb{R}}$, which satisfy the following initial
value problem
\[
\dot{t}=-2\tau,\quad\dot{x}=2\eta,\quad\dot{\tau}=0,\quad\dot{\eta}=0,\qquad
t(0)=0,\quad x(0)=x,\quad\tau(0)=\pm|\eta|,\quad\eta(0)=\eta\;.
\]
Note that the initial condition for $\tau$ is inferred from the fact that null
bicharacteristics are in the characteristic variety. One can use $t$ as a
parameter and obtain the following solutions
\begin{equation}
\label{E:bic}\gamma_{+}(t)=(t,x+t\eta/|\eta|;-|\eta|,\eta)\quad\mbox{and}\quad
\gamma_{-}(s)=(t,x-t\eta/|\eta|;|\eta|,\eta)\;,\qquad t\in{\mathbb{R}}\;.
\end{equation}
Note that $\gamma_{+}$ has the direction vector $(|\eta|,\eta)$ for the first
$d+1$ components where as $\gamma_{-}$ has the direction $(|\eta|,-\eta)$.

Since we are interested in the forward problem, we will restrict ourselves to
$t\geq0$. The wave front set of $u$ consists exactly of all bicharacteristics
whose initial data are in the wave front set of $f$, see for example
\cite[Chapter XXIII]{hormander85}. The next result describes the wave front
set of the trace of the wave $u$ along $\partial\Omega$.

\begin{prop}
\label{T:Ellera} Let $f\in H^{s}(\Omega_{0})$ for some $s\in{\mathbb{R}}$ and
let $u\in C([0,\infty),H^{s}({\mathbb{R}}^{d}))$ be the unique solution to
(\ref{E:originalwave})-(\ref{E:originalBC}) and $g=u\Big|_{{\mathbb{R}}%
_{+}\times\partial\Omega}$. Then,
\[%
\begin{split}
\mathrm{WF}(g)\subset &  \{(|y-x|,y;\mp\alpha|y-x|,\pm\alpha[y-x-((y-x)\cdot
\nu_{y})\nu_{y}])\in T^{\ast}({\mathbb{R}}_{+}\times\partial\Omega
)\setminus0\;:\\
&  y\in\partial\Omega,\alpha>0\mbox{ and }(x,\pm\lbrack y-x])\in
\mathrm{WF}(f)\}\;.
\end{split}
\]
Furthermore, every point in $\mathrm{WF}(g)$ is a hyperbolic point. This is to
say that $(t,y;\tau,\xi)\in\mathrm{WF}(g)$ implies $|\tau|>|\xi|$.
\end{prop}

\begin{proof}
We know that
\begin{equation}
\label{E:Eller8}\mathrm{WF}(u) =\{ (t,x\pm t \eta/|\eta|, \mp|\eta|, \eta)
\;:\; t\ge0 \mbox{ and } (x,\eta) \in\mathrm{WF}(f)\}\subset T^{*}{\mathbb{R}%
}^{d+1}_{+} \setminus0
\end{equation}
and observe that $\mathrm{WF}(u) \cap N({\mathbb{R}}_{+} \times\partial
\Omega)=\emptyset$. Hence, the previous lemma is applicable.

If $y\in\partial\Omega$, then $y=x\pm t\eta/|\eta|$ implies that $\pm
\alpha(y-x)=\eta$ for some $\alpha>0$ and $t=|y-x|$. Hence,
\begin{multline*}
\mathrm{WF}(u)\cap T_{{\mathbb{R}}_{+}\times\partial\Omega}^{\ast}{\mathbb{R}%
}^{d+1}\\
=\{(|y-x|,y;\mp\alpha|y-x|,\pm\alpha(y-x))\;:\;y\in\partial\Omega
,\alpha>0\mbox{ and }(x,\pm\lbrack y-x])\in\mathrm{WF}(f)\}
\end{multline*}
Now we apply Lemma~\ref{L:trace} and the proof of the first statement is finished.

For the second statement, let $(t,y;\tau,\xi)\in\mathrm{WF}(g)$. Then there
exists a point $(t,y;\tau,\eta)\in\mathrm{WF}(u)$ where
\[
\xi=\eta-(\eta\cdot\nu_{y})\nu_{y} \quad\mbox{ and }\quad|\tau|=|\eta|\;.
\]
Since the support $\Omega_{0}$ is a compact subset of the convex set $\Omega$, no
bicharacteristic tangential to ${\mathbb{R}}_{+}\times\partial\Omega$ can meet
$\Omega_{0}$. Hence, $\eta\notin T_{y}^{\ast}\mathbb{S}^{d-1}$ and thus $|\xi
|<|\eta|=|\tau|$.
\end{proof}

\begin{remark}
The inclusion in Proposition \ref{T:Ellera} can be replaced by an equal sign.
This follows from the fact that the operator $\mathcal{A}$ is a sum of two elliptic
Fourier integral operators.
\end{remark}

Now comes the second step. The unique solution $w(t,x)$ to the exterior
problem (\ref{E:Eller2}) has the following interesting property. The wave
fronts of $u$ and $w$ are the same in the exterior upper half space
\[
D_{+}=\{x\in\mathbb{R}^{d} \;:\;x_{d}\ge-\delta\}\;.
\]
More precisely, we have
\begin{theorem}
\label{T:Ellerb}
For $f\in H^s_0(\Omega_0)$, suppose that $u$ is the solution to the Cauchy problem (\ref{E:originalwave})-(\ref{E:originalBC}) and that $w$ is the solution to the exterior initial-boundary value problem (\ref{E:Eller2}). Then
\[
 \mathrm{WF}(u-w)\cap T^{\ast} \{(t,x)\in Q^\mathrm{ext}\;:\; x_d\ge -\delta\}= \emptyset
\]
and
\[
 \mathrm{WF}(u-w)(2,\cdot) \cap T^{\ast}D_+= \emptyset\;.
\]
\end{theorem}
\begin{proof}
Let $v$ be the solution of the exterior Cauchy-Dirichlet problem with zero
initial data and Dirichlet data $(1-\psi)g$, i.e. $v=u-w$. We will show that
\[
\mathrm{WF}(v)\cap T^{\ast}\{(t,x)\in Q^\mathrm{ext}\;:\; x_d\ge -\delta\}=\emptyset\;,
\]
i.e. the function $v$ is infinitely often differentiable in the upper half
space. According to Theorem 2.1 \cite{ta78}, the singularities of $v$ depend
on the wave front set of the Dirichlet data $(1-\psi)g$. Since $\psi\in
C^{\infty}$ we know that $\mathrm{WF}((1-\psi)g)\subset\mathrm{WF}(g)\cap T^{\ast}({\mathbb{R}}%
_{+}\times \mathbb{S}_{-}^{d-1})$, where $\mathbb{S}_{-}^{d-1}=\{\omega\in \mathbb{S}^{d-1}%
\;:\;\omega_{d}<-\delta\}$.

According to Proposition \ref{T:Ellera} the wave front
set of $g$ depends only on the wave front set of $f$. In particular all points
in $\mathrm{WF}(g)$ are hyperbolic points. Over each hyperbolic points there
are exactly four bicharacteristics, two incoming ones and an two outgoing ones. Since the initial data are zero, the solution is identically zero along all incoming bicharacteristics.

Let now $(t,y;\tau,\xi)\in\mathrm{WF}(g)$ and denote the exterior unit normal vector to $\Omega$ at $y$ by $\nu_y$. Then, according to Proposition
\ref{T:Ellera}, we have either $\tau>0$ (bicharacteristic $\gamma_{-}$) or
$\tau<0$ (bicharacteristic $\gamma_{+}$).  The spatial frequency of the outgoing bicharacteristic is $\eta = \xi + \sqrt{\tau^2-|\xi|^2}\nu_y$ in the case $\gamma_+$ and $\eta = \xi-\sqrt{\tau^2-|\xi|^2}\nu_y$ in the case of $\gamma_-$, compare with formula(\ref{E:bic}).

We will consider only the case $\tau>0$, the case $\tau<0$ is similar. Because of Proposition \ref{T:Ellera} there exists a point $x\in
\Omega_{0}$ and $\alpha>0$ such that
\[
t=|y-x|,\quad\tau=\alpha|y-x|,\quad\xi=-\alpha[y-x-((y-x)\cdot \nu_y)\nu_y]\;.
\]
Compute now
\[
\tau^{2}-|\xi|^{2}=\alpha^{2}[(y-x)\cdot \nu_y]^{2}
\]
and observe that the convexity of $\Omega$ implies that
\[
 \sqrt{\tau^{2}-|\xi|^{2}}=\alpha[(y-x)\cdot \nu_y]\;.
\]
The outgoing bicharacteristic $\gamma_-$ starting at the point $(t,y)$ must have the
direction
\[
(\tau,\sqrt{\tau^{2}-|\xi|^{2}}\nu_y-\xi)=\alpha(|y-x|,y-x)\;,
\]
compare with (\ref{E:bic}). Since $x\in\Omega_{0}$ and $y\in \mathbb{S}_{-}^{d-1}$, the projection of the bicharacteristic $\gamma_-$ into space-time is a
straight line which cannot meet $\{(t,x)\in Q^\mathrm{ext}\;:\;x_d\ge -\delta\}$ for $t>0$.

The second statement follows from the first one and Lemma \ref{L:trace}
\end{proof}

\section{The Radon transform and its wave front set}\label{S:Radon}

We have defined the Radon transform earlier, by equation (\ref{E:Radontr}).

\begin{prop}
\label{T:Ellerc} The Radon transform can be extended to a bounded linear
operator from $\mathcal{E}^{\prime}({\mathbb{R}}^{d})$ to $\mathcal{D}%
^{\prime}(\mathbb{S}^{d-1}\times{\mathbb{R}})$ and
\begin{equation}%
\begin{split}
\mathrm{WF}(\mathcal{R}f)\subset\{  &  (\pm\omega,\pm\omega\cdot x;\mp
\alpha\lbrack x-(x\cdot\omega)\omega],\pm\alpha)\in T^{\ast}(\mathbb{S}^{d-1}%
\times{\mathbb{R}})\setminus0\\
&  \;:\omega\in \mathbb{S}^{d-1},\alpha>0\mbox{ and }(x,\omega)\in\mathrm{WF}(f)\}\;.
\end{split}
\label{E:Eller7}%
\end{equation}

\end{prop}

For $d=2$ this is Theorem 3.1 in \cite{qu93}.

\begin{proof}
A large part of the proof is an application of Theorem 8.2.13 in
\cite{hormander83a}. We use local coordinates $\omega=\varphi(y^{\prime})$,
$y^{\prime}\in Y^{\prime}\subset{\mathbb{R}}^{d-1}$ where we demand that
$\varphi(Y^{\prime})$ is at most a hemisphere. The kernel of the Radon
transform is
\[
K(y,x)=\delta(\varphi(z)\cdot x-p)\;,
\]
where we abbreviate $y=(y^{\prime},p)\in{\mathbb{R}}^{d}$ and see that the kernel is a
composition of the delta function and the function $m(y,x)=\varphi(y^{\prime})\cdot
x-p$. Note that $K\in\mathcal{D}^{\prime}(Y\times X)$ where $Y=Y^{\prime
}\times{\mathbb{R}}$ and $X={\mathbb{R}}^{d}$. Recall that $\mathrm{WF}%
(\delta)=\{(0,\tau)\in T^{\ast}{\mathbb{R}}\setminus0\;:\;\tau\in{\mathbb{R}%
}\}$. The wave front set of this kernel is analyzed using Theorem 8.2.4 in
\cite{hormander83a}. We have
%\[
%m'(y,x) = \Mat{ \varphi'(y')^{\top} x \\ -1 \\ \varphi(y') }
%\]%
\[
m^{\prime}(y,x)=\left[
\begin{array}
[c]{c}%
\varphi^{\prime}(y^{\prime})^\top x\\
-1\\
\varphi(y^{\prime})
\end{array}
\right]
\]
and hence,
%\begin{equation}\label{E:Eller6}
%{\rm  WF}(K) \subset \left\{ \left(\Mat{y' \\ \langle \varphi(y'),x\rangle},x; \Mat{ \tau \varphi'(y')^{\top} x \\ -\tau} , \tau \varphi(y')\right)\in T^*(Y \times X)\setminus 0\;:\;
%\tau\in {\mathbb{R}}, y'\in Y', x\in {\mathbb{R}}^d\right\}\;.
%\end{equation}%
\begin{equation}
\mathrm{WF}(K)\subset\left\{  \left(  \left[
\begin{array}
[c]{c}%
y^{\prime}\\
\varphi(y^{\prime})\cdot x
\end{array}
\right]  ,x;\left[
\begin{array}
[c]{c}%
\tau\varphi^{\prime}(y^{\prime})^{\top}x\\
-\tau
\end{array}
\right]  ,\tau\varphi(y^{\prime})\right)  \in T^{\ast}(Y\times X)\setminus
0\;:\;\tau\in{\mathbb{R}},y^{\prime}\in Y^{\prime},x\in{\mathbb{R}}%
^{d}\right\}  \;. \label{E:Eller6}%
\end{equation}

Here $\varphi^{\prime}$ is the Jacobian matrix of $\varphi$. Now Theorem
8.2.13 in \cite{hormander83a} is used to finish the proof. At first note that
\[
\mathrm{WF}^{\prime}(K)_{X}=\{(x,\xi)\in T^{\ast}X\setminus0\;:\;(y\cdot
x,x;0,-\xi)\in\mathrm{WF}(K)\mbox{ for some }y\in Y\}=\emptyset\;,
\]
since the first fiber component in (\ref{E:Eller6}) cannot be zero. Hence, the
Radon transform can be continuously extended to a linear operator from
$\mathcal{E}^{\prime}(X)$ into $\mathcal{D}^{\prime}(Y)$. Similarly, we have
\[
\mathrm{WF}(K)_{Y}=\{(y,\eta)\in T^{\ast}X\setminus0\;:\;(y,x;\eta
,0)\in\mathrm{WF}(K)\mbox{ for some }y\in Y\}=\emptyset\;,
\]
and hence,
\begin{equation}%
\begin{split}
&  \mathrm{WF}(\mathcal{R}(f))\subset\mathrm{WF}^{\prime}(K)\circ
\mathrm{WF}(f)=\{(y,\eta)\;:\;(y,x;\eta,-\xi)\in\mathrm{WF}%
(K)\mbox{ for some }(x,\xi)\in\mathrm{WF}(f)\}\\
&  =\left\{  \left(  \left[
\begin{array}
[c]{c}%
y^{\prime}\\
\varphi(y^{\prime})\cdot x
\end{array}
\right]  ,\left[
\begin{array}
[c]{c}%
-\tau\varphi^{\prime}(y^{\prime})^{\top}x\\
\tau
\end{array}
\right]  \right)  \in T^{\ast}(Y)\setminus0\;:\;\tau\in{\mathbb{R}},y^{\prime
}\in Y^{\prime},(x,\varphi(y^{\prime}))\in\mathrm{WF}(f)\right\}  \;.
\end{split}
\label{E:Eller10}%
\end{equation}
Finally we return from the local coordinate system in $Y$ to $(\omega,p)\in
\mathbb{S}^{d-1}\times{\mathbb{R}}$. Here it is important to note that the co-vector
$\varphi^{\prime}(y^{\prime})^{\top}x\in T_{y^{\prime}}^{\ast}Y^{\prime}$ gets
mapped to $x-(\omega\cdot x)\omega\in T_{\omega}^{\ast}\mathbb{S}^{d-1}$. Furthermore,
if $(x,\omega)\in\mathrm{WF}(f)$, then $(x,\tau\omega)\in\mathrm{WF}(f)$ for
$\tau>0$ and $(x,-\tau\omega)\in\mathrm{WF}(f)$ for all $\tau<0$.
\end{proof}

Actually, a more precise result holds true, there is equality in
(\ref{E:Eller7}). This can be inferred from the theory of Fourier integral
operators and is explained in \cite[p. 59]{bal19}.

The Radon transforms of the solutions $u(t,x)$ to the Cauchy problem
(\ref{E:originalwave})-(\ref{E:originalBC}) and $w(t,x)$ to the
Cauchy-Dirichlet problem (\ref{E:Eller2}) carry $t$ as a parameter. Recall
$\mathbb{S}_{+}^{d-1}=\{\omega\in \mathbb{S}^{d-1}\;:\;\omega_{d}>-\delta/2\}$.

\begin{theorem}\label{T:main}
Let $f\in H^s_0(\Omega_0)$ and suppose that $u$ is the solution to the Cauchy problem (\ref{E:originalwave})-(\ref{E:originalBC}). Let $w$ be the solution to the exterior initial-boundary value problem (\ref{E:Eller2}). Then
\[
\mathrm{WF}(\left[  \mathcal{R}(u - w)\right]  (2,\cdot)) \cap T^{\ast}(\mathbb{S}_{+}
^{d-1}\times(1,3))= \emptyset \;.
\]
\end{theorem}
\begin{proof}
As before, let $v=u-w$. We have to show that $\mathrm{WF}(\left[
\mathcal{R}v\right]  (2,\cdot))\cap T^{\ast}(\mathbb{S}_{+}^{d-1}\times(1,3))=\emptyset$. Let $\omega\in \mathbb{S}_{+}^{d-1}$. Proposition \ref{T:Ellerc} can be
rephrased as follows: The wave front set of the Radon transform of
$v(2,\cdot)$ can have points with base point $(\omega,p)$ only if the wave
front set of $v(2,\cdot)$ has points of the form $(x,\pm\omega)$ for some
$x\in{\mathbb{R}}^{d}$ satisfying $\omega\cdot x=p$. According to Theorem
\ref{T:Ellerb}, the wave front sets of $v(2,\cdot)$ is empty over the set $D_+$
whereas over $D_-:=\{ x\in \mathbb{R}^d\setminus \overline{\Omega}\;:\; x_d<-\delta\}$ the wave front set of $v(2,\cdot)$ is a subset of the wave front
set of $u(2,\cdot)$.

However, we can show that the wave front set of $u$ over $D_{-}$ has no points of the form
$(x,\pm\omega)$ with $x$ satisfying $\omega\cdot x=p$ for some $p>-1$. The wave
front set of $u$ is given in (\ref{E:Eller8}). If $(x,\pm\omega)\in
\mathrm{WF}\left(  u(2,\cdot)\right)  $, then $(y,\pm\omega)\in\mathrm{WF}(f)$
where $x=y\pm 2\omega$. Note that $x=y+2\omega$ is not possible since
$\mathrm{supp}\,f=\Omega_{0}$ implies $y_d\ge 0$ and thus $x_d>-\delta$. Hence, $x\cdot\omega=y\cdot\omega-2<-1$.
\end{proof}
Recall the definitions of $\tilde{F}$ and $G$ given in (\ref{E:Ftilde}) and (\ref{E:G}). We will be interested in the error
\[
 E_G(\omega,p) = F(\omega,p) - G(\omega,p)\;.
\]
For $\omega\in \mathbb{S}_{+}^{d-1}$ and $p \in (-1,1)$ we have just proved that $F(\omega,p)-\tilde{F}(\omega,p)$ is an infinitely smooth function. Furthermore, by symmetry, for $\omega_d <\delta/2$ and $p\in (-1,1)$, we infer that $F(\omega,p) - \tilde{F}(-\omega,-p)$ is smooth.

Let $H$ denote the Heaviside function of a one-dimensional variable
\[
H(s)\equiv\left\{
\begin{array}
[c]{cc}%
1, & s\geq0,\\
0, & s<0.
\end{array}
\right.
\]
Then,
\begin{equation}\label{E:Eller9a}
 E_G(\omega,p) = F(\omega) - H(\omega_d)\tilde{F}(\omega,p) -[1-H(\omega_d)] \tilde{F}(-\omega,-p)
\end{equation}
and this function defined on
$\mathbb{S}^{d-1}\times (-1,1)$ is an infinitely smooth function except for a
discontinuity on the set $\mathbb{E}\times (-1,1)$ where $\mathbb{E}$ is the equator $\mathbb{E}\equiv\{ \omega  \;:\;\omega\in \mathbb{S}^{d-1},\omega_d=0\}$. The next result shows that these singularities disappear, once the inverse Radon transform is applied. In order to apply the inverse Radon transform we need to extend $E_G$ to all $p\in \mathbb{R}$. This is done by setting $G\equiv 0$ for all $|p|>1$. Abusing notation, we will denote the resulting extension by $E_G$ as well.

\begin{prop}\label{T:final}
The inverse Radon transform $\mathcal{R}^{-1}E_G$ restricted to the open unit ball is a smooth function $\kappa\in C^{\infty}(B(0,1))$.
\end{prop}

\begin{proof}
 The function $E_G$ has two locations of singularities: One is on $\mathbb{E}\times (-1,1)$, due to (\ref{E:Eller9a}) and the other is at $|p|=1$ due to the extension by zero outside of $p\in [-1,1]$.

We discuss at first $\mathrm{WF}(E_G) \cap T^*(\mathbb{S}^{d-1}\times (-1,1))$.  According to formula (\ref{E:Eller9a}), we have
\[
 \begin{split}
\mathrm{WF}(E_G)\cap T^*(\mathbb{S}^{d-1}\times (-1,1))\subset &\left\{ \left(
 \left[ \begin{matrix} \omega \\ p \end{matrix} \right],
 \left[ \begin{matrix} \xi \\  0 \end{matrix} \right] \right)
   \in T^{\ast}(\mathbb{S}^{d-1}\times (-1,1))\setminus
0\;: \right.  \\& \left. \omega_{d}=0, \xi \perp \dot{c}(0), c(0)=\omega, \mbox{ for all } c(t) \in \mathbb{E} \right\}  \;.
 \end{split}
\]
Here $c:(-\alpha,\alpha)\to \mathbb{E}$ denotes regular curves in $\mathbb{E}$.
In view of (\ref{E:Eller7}) we see that this set has an empty intersection with the wave front set of the Radon transform of any distribution.

From the theory of elliptic Fourier integral operators
\cite[Theorem 4.2.5]{duistermaat96} one knows that the inverse transform
$\mathcal{R}^{-1}:\mathcal{E}^{\prime}({S}^{d-1}\times\mathbb{R})\rightarrow
\mathcal{D}^{\prime}(\mathbb{R}^{d})$ has the same canonical relation as does
$\mathcal{R}$, only with the variables switched around. The application of the inverse Radon transform will ignore those points in the
wave front set of $E_G$.

In view of (\ref{E:Eller7}), the singularities of $E_G$ at $|p|=1$ will correspond to singularities of its inverse Radon transform outside of the unit ball.

But this means that $\mathrm{WF}(\mathcal{R}^{-1}[E_G])\cap T^*B(0,1)=\emptyset$.
\end{proof}

\section{A fast algorithm for the reduced circular geometry in 2D}

\label{S:algorithm}

The three-step image reconstruction procedure proposed in this paper is
detailed at the end of section~\ref{S:reduced-reco}. Practicality of this
technique depends, in particular, on the computational complexity of steps 1
and 2. (Efficient algorithms for implementing step 3 are well known). Let us
assume that the domain $\Omega$ is discretized using $m^{d}$ nodes, and the
number of time steps is $\mathcal{O}(m)$. The optimal number of projections is
then $\mathcal{O}(m^{d-1})$ with the number of samples in a projection equal
to $\mathcal{O}(m).$ If the first step is implemented using finite
differences, it will have computational complexity $\mathcal{O}(m^{d+1})$  floating point operations (flops). A
straightforward computation of the Radon projections requires $\mathcal{O}%
(m^{d+2})$ flops.

However, for special geometries it might be possible to design significantly
faster algorithms. In the present section we develop an asymptotically fast
($\mathcal{O}(m^{2}\log m)$ flops) algorithm applicable in 2D when the
acquisition curve is a circular arc.

To be precise, we consider the following geometry. The region $\Omega$ is the
unit disk $B(0,1)$ centered at the origin, $\Omega_{0}$ is the upper half of
$B(0,1),$ i.e.%
\[
\Omega_{0}\equiv\{x:|x|<1\text{ and }x_{d}>0\}.
\]
The data acquisition surface $\Gamma$ is an arc of a unit circle%
\[
\Gamma\equiv\{z:|z|=1\text{ and }z_{d}>-2\delta.\}
\]
The geometry is shown in Figure~\ref{F:circgeom}, together with the phantom we
use in the numerical simulations described further below (for which we chose
$2\delta$ to be equal to $\arctan(10^{\circ})$).
\begin{figure}[h]
\begin{center}
\includegraphics[scale = 0.9]{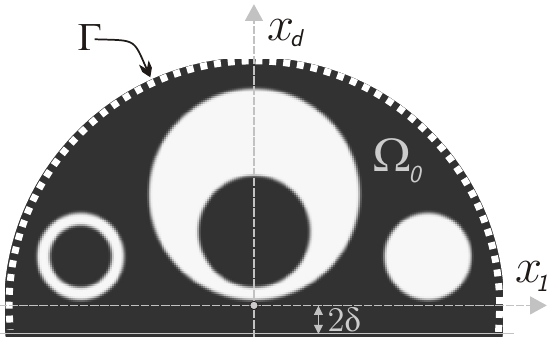}
\end{center}
\caption{Reduced circular data acquisition geometry}%
\label{F:circgeom}%
\end{figure}

Due to the circular symmetry of $\Omega$, solving the exterior boundary value
problem~(\ref{E:Eller2}) commutes with rotation. In addition, it commutes with
the shifts in time variable. It is well known that the Radon transform
commutes with rotation, too. Thus, one can expect that the operator that maps
boundary data $h(\omega,t)$ into the Radon projections is a convolution, and
that the Fourier transform in $t$ combined with Fourier series in the angular
variable would diagonalize it. This is indeed the case as shown below.

We start by noticing that the wave equation separates in the polar coordinate
system. In particular, the following combinations $W_{k}(t,\lambda,r,\theta)$
of functions represent radiating solutions of the wave equation (propagating
away from the origin):
\[
W_{k}(t,\lambda,r,\theta)\equiv\frac{H_{|k|}^{(1)}(\lambda r)\exp
(ik\theta)\exp(-i\lambda t)}{H_{|k|}^{(1)}(\lambda)},\quad\lambda\in
\mathbb{R}\backslash0,\quad k\in\mathbb{Z}.
\]
where $H_{k}^{(1)}(s)$ is the Hankel function of the first kind of order $k$,
$\lambda$ is any non-zero frequency, and $r$ and $t$ are the polar coordinates
corresponding to the variable $x,$ i.e. $x=r(\cos\theta,\sin\theta)$. Let us
extend $W_{k}(t,\lambda,r,\theta)$ to $\lambda=0$ by continuity (see
\cite{NIST}, section 10.7)%
\[
W_{k}(t,0,r,\theta)\equiv\frac{1}{r^{|k|}}\exp(ik\theta),\quad k\in
\mathbb{Z}.
\]
At $\lambda=0$ functions $W_{k}(t,0,r,\theta)$ are stationary solutions of the
wave equations (i.e., harmonic functions), bounded at infinity. The values of
$W_{k}$'s on the cylinder $r=1,$ $\theta\in\lbrack0,2\pi],$ $t\in\mathbb{R}$
are functions $\exp(ik\theta)\exp(-i\lambda t)$; they represent a Fourier basis.

We will use functions $W_{k}(t,\lambda,r,\theta)$ to solve the forward problem
(\ref{E:Eller2}). Let us extend the boundary condition $h(t,x(1,\theta))$ in
the time variable by 0 to all $t<0.$ Since solutions of the wave equation are
causal, and we are only interested in $w(t,x)$ at $t=T,$ we can smoothly
cut-off $h(t,x(1,\theta))$ to $0$ on the interval $[T,T+b],$ and define
$h(t,x(1,\theta)$ to be 0 for all $t>T+b.$ There is enough freedom in defining
this cut-off to enforce the condition%
\[
\int\limits_{0}^{2\pi}\int\limits_{\mathbb{R}}h(t,x(1,\theta))dtd\theta=0.
\]
Now, solution of problem (\ref{E:Eller2}) can be formally represented as
\begin{equation}
w(t,x)\equiv w^{\ast}(t,r,\theta)=\int\limits_{\mathbb{R}}\left[
\sum_{k=-\infty}^{\infty}b_{k}(\lambda)W_{k}(t,\lambda,r,\theta)\right]
d\lambda. \label{E:separation}%
\end{equation}
where coefficients $a_{k}(\lambda)$ are to be determined from the boundary
condition%
\begin{equation}
w^{\ast}(t,1,\theta)=h(t,x(1,\theta)), \label{E:circleBC}%
\end{equation}
on $\mathbb{R}\times\mathbb{S}^{1},$ which yields%
\begin{equation}
b_{k}(\lambda)=\frac{1}{(2\pi)^{2}}\int\limits_{0}^{2\pi}\left[
\int\limits_{\mathbb{R}}h(t,x(1,\theta))\exp(i\lambda t)dt\right]
\exp(-ik\theta)d\theta. \label{E:Fourier-coefs}%
\end{equation}
We notice, for future use, that each $b_{k}(\lambda)$ is continuous (in fact,
real-analytic) due to the bounded support of $h$ in time variable.

From the numerical standpoint it is somewhat easier to find instead of the
Radon projections their derivative in $p$. To that end we will use
$\frac{\partial}{\partial t}w^{\ast}(t,r,\theta)$; we only need its value at
$t=T$:%
\begin{align}
\frac{\partial w}{\partial t}(T,x)  &  =\frac{\partial w^{\ast}}{\partial
t}(T,r,\theta)=-i\int\limits_{\mathbb{R}}\left[  \sum_{k=-\infty}^{\infty
}b_{k}(\lambda)\lambda W_{k}(T,\lambda,r,\theta)\right]  d\lambda.\nonumber\\
&  =-i\int\limits_{\mathbb{R}\backslash0}\left[  \sum_{k=-\infty}^{\infty
}b_{k}(\lambda)\lambda\frac{H_{|k|}^{(1)}(\lambda r)}{H_{|k|}^{(1)}(\lambda
)}\exp(ik\theta)\exp(-i\lambda T)\right]  d\lambda. \label{E:dwdt}%
\end{align}
In the above equation we are able to exclude $\lambda=0$ from the integration
domain since the integrand is zero at this point.

Our goal is to approximate $\frac{\partial}{\partial p}F(\omega,p)$ by the
Radon transform of $\frac{\partial w}{\partial t}$:%
\[
\frac{\partial}{\partial p}F(\omega,p)\thickapprox-2\left[  \mathcal{R}%
\frac{\partial w}{\partial t}\right]  (2,\omega,p+2)
\]
However, computing $\frac{\partial w}{\partial t}(T,x)$ by using equation
(\ref{E:dwdt}) and by numerically evaluating its Radon transform would be
relatively slow. Instead, we take advantage of the fact that the line
integrals of the circular wave functions $H_{|k|}^{(1)}(\lambda_{l}%
r)\exp(ik\theta)$ can be found analytically as closed form expressions.

Define $I_{k}(\omega(\alpha),p,\lambda),$ with $\omega(\alpha)=(\cos
\alpha,\sin\alpha)$ as the value of the Radon transform of the function
$\frac{H_{|k|}^{(1)}(\lambda r(x))}{H_{|k|}^{(1)}(\lambda)}\exp(ik\theta(x)$
at $(\omega(\alpha),p)$:
\begin{align*}
I_{k}(\omega(\alpha),p,\lambda) &  =\mathcal{R}\left[  \frac{H_{|k|}%
^{(1)}(\lambda r(x))}{H_{|k|}^{(1)}(\lambda)}\exp(ik\theta(x)\right]
(\omega(\alpha),p)\\
&  =\frac{1}{H_{|k|}^{(1)}(\lambda)}\int\limits_{\mathbb{R}}H_{|k|}%
^{(1)}(\lambda r(x))\exp(ik\theta(x))\delta(x\cdot\omega(\alpha)-p)dx.
\end{align*}
Compute first
\[
H_{|k|}^{(1)}(\lambda)I_{|k|}(\omega(0),p,\lambda)=2\int\limits_{0}^{\infty
}H_{|k|}^{(1)}\left(  \lambda\sqrt{p^{2}+x_{2}^{2}}\right)  \cos\left(
k\arcsin\left(  \frac{x_{2}}{r}\right)  \right)  dx_{2}.
\]
Using formulas 6.941.3 and 6.941.4 in \cite{Gradshteyn}, one obtains, for
$\lambda>0,$
\[
I_{k}(\omega(0),p,\lambda)=2\frac{(-i)^{|k|}}{\lambda H_{|k|}^{(1)}(\lambda
)}\exp(i\lambda p).
\]
Then, for general $\alpha,$
\begin{equation}
I_{k}(\omega(\alpha),p,\lambda)=2\frac{(-i)^{|k|}}{\lambda H_{|k|}%
^{(1)}(\lambda)}\exp(i\lambda p)\exp(ik\alpha).\label{E:fancy-int}%
\end{equation}
By utilizing the well-known formulas for Hankel functions of a negative
argument (see \cite{NIST}, section 10.11.7)%
\[
H_{|k|}^{(1)}(-s)=(-1)^{|k|-1}H_{|k|}^{(2)}(s)=(-1)^{|k|-1}\overline
{H_{|k|}^{(1)}(s)},
\]
one verifies that (\ref{E:fancy-int}) is valid for all $\lambda\neq0.$

In order to turn the above relations into an algorithm, let us assume that $M$
point-like transducers are placed at equispaced point $\theta_{j}%
=j\Delta\theta,$ $j=0,1,2,..M,$ and the measurements are done on an equispaced
time grid $t_{i}=i\Delta t,$ $i=1,2,...,N.$ We approximate $b_{k}(\lambda)$ on
a uniform grid of frequencies $\lambda_{l}=l\Delta\lambda,$ $l=-l_{Nyq}%
,...,l_{Nyq}$, where $l_{Nyq}\Delta\lambda$ equals to the Nyquist frequency of
the discretization in time. This is can be done by discretizing
(\ref{E:Fourier-coefs}) using the trapezoid rule, and applying the Fast
Fourier Transform (FFT)\ algorithm both in $\theta$ and in $t.$\ Such an
approximation of the outer integral of a periodic function in
(\ref{E:Fourier-coefs}) is quite standard. However, the use of the FFT in $t$
is equivalent to making the boundary condition periodic in time. In order to
reduce\ the error introduced by this effect, the computation of the inner
integral should be done with enough zero-padding. On the other hand, computing
$b_{k}$'s using the 2D FFT makes this step asymptotically (and practically)
very fast.

Once coefficients $b_{k}(\lambda_{l})$ are obtained, one can approximate the
time derivative of $w^{\ast}(t,r,\theta)$ at $t=T$:%
\begin{equation}
\frac{\partial w^{\ast}}{\partial t}(T,r,\theta)\approx-\sum
_{\substack{l=-l_{Nyq},\\l\neq0}}^{l_{Nyq}}\sum_{k=-M/2}^{M/2-1}i\lambda
_{l}b_{k}(\lambda_{l})W_{k}(T,\lambda_{l},r,\theta). \label{E:approx_sol}%
\end{equation}
\bigskip

Combining this result with (\ref{E:approx_sol}) and substituting $T=2$ we
obtain:%
\begin{align}
\frac{\partial}{\partial p}\mathcal{R[}w(2,x)](\omega(\alpha),p+2)  &
\approx\sum_{k=-M/2}^{M/2-1}\left[  \sum_{\substack{l=-l_{Nyq},\\l\neq
0}}^{l_{Nyq}}\frac{2i(-i)^{|k|}b_{k}(\lambda_{l})}{H_{|k|}^{(1)}(\lambda_{l}%
)}\exp(i\lambda_{l}p)\right]  \exp(ik\alpha)\nonumber\\
&  =\sum_{k=-M/2}^{M/2-1}\left[  \sum_{l=-l_{Nyq},}^{l_{Nyq}}a_{k,l}%
\exp(i\lambda_{l}p)\right]  \exp(ik\alpha), \label{E:final-alg}%
\end{align}
with
\begin{equation}
a_{k,l}=\left\{
\begin{array}
[c]{cc}%
\frac{2i(-i)^{|k|}b_{k}(\lambda_{l})}{H_{|k|}^{(1)}(\lambda_{l})}, & l\neq0\\
0, & l=0
\end{array}
\right.  . \label{E:a-coefs}%
\end{equation}
One easily recognizes that formula (\ref{E:final-alg}) is a 2D discrete
Fourier transform, and thus it can be efficiently implemented using the 2D FFT. Our
algorithm then consists in

\begin{enumerate}
\item Computing coefficients $b_{k}(\lambda_{l})$ (equation
(\ref{E:Fourier-coefs}) using 2D FFT;

\item Computing coefficients $a_{k,l}$ using (\ref{E:a-coefs});

\item Computing $\frac{\partial}{\partial p}\mathcal{R[}w(2,x)](\omega
(\alpha),p+2)$ given by (\ref{E:final-alg})\ using 2D FFT;

\item Computing $\frac{\partial}{\partial p}F(\omega,p)\approx2\frac{\partial
}{\partial p}\mathcal{R[}w(2,x)](\omega(\alpha),p+2)$ for $\alpha\in
\lbrack0,\pi]$;

\item Computing $\frac{\partial}{\partial p}F(\omega,p)\approx-\frac{\partial
}{\partial p}F(-\omega,-p)$ for $\alpha\in(\pi,2\pi)$.
\end{enumerate}

An optional sixth step is to compute $F(\omega,p)$ from $\frac{\partial
}{\partial p}F(\omega,p)$ by numerical anti-differentiation, while enforcing
conditions $F(\omega,1)=F(\omega,-1)=0$ in the least square sense. This step
is optional since in the formula (\ref{E:Inverse_Radon}) (first raw)\ the
Hilbert transform commutes with differentiation, and therefore (\ref{E:Inverse_Radon})
can be easily modified to reconstruct $f(x)$ from $\frac{\partial}{\partial
p}F(\omega,p).$

If one assumes (realistically) that $M$, $N,$ and $l_{Nyq}$ are of order of
$\mathcal{O}(m),$ the algorithm requires $\mathcal{O}(m^{2}\log n)$ flops,
since the most time consuming operations here are the 2D FFT's.

The algorithm we presented shares important computational steps with the 2D
algorithm presented in \cite{Kun-cyl} and with the method applicable to the
reduced circular geometry in \cite{Kun-Do}. We would like to point out that
the technique of \cite{Kun-cyl} was developed to solve the inverse source
problem with full data; reduced data were not considered. The algorithm of
\cite{Kun-Do} yields a theoretically exact reconstruction from reduced data,
but at the expense of a significantly reduced support of source term $f(x)$ as
compared to the theoretically visible region. In the present method, the
theoretically visible region $\{x:x\in B(0,1)\}$ and $x_{2}>-2\delta\}$ is
only slightly larger than $\Omega_{0},$ since $\delta$ can be chosen to be much smaller than
1. However, the present technique is not theoretically exact.

\begin{figure}[t]
\begin{center}
\subfigure[Initial condition $f(x)$]{
\includegraphics[scale = 0.15]{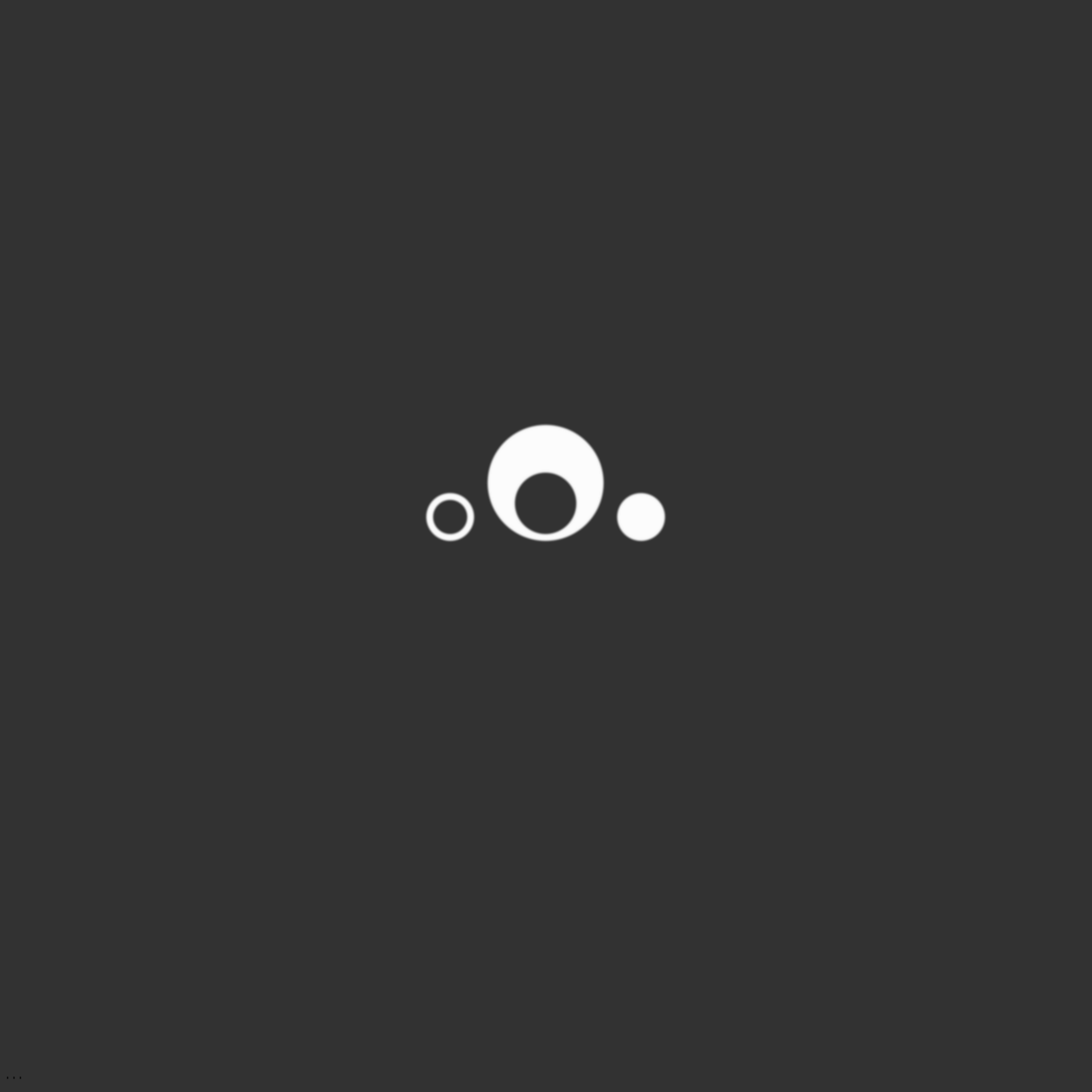} \ \includegraphics[scale = 0.65]{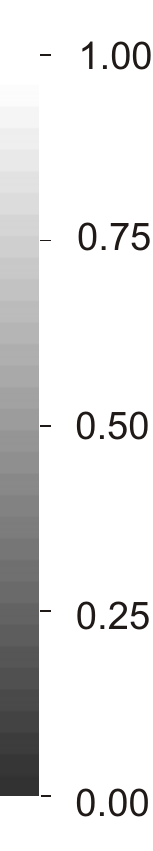}}
\subfigure[Exterior solution $u(t,x)$ at $t=2$ ]{
\includegraphics[scale = 0.15]{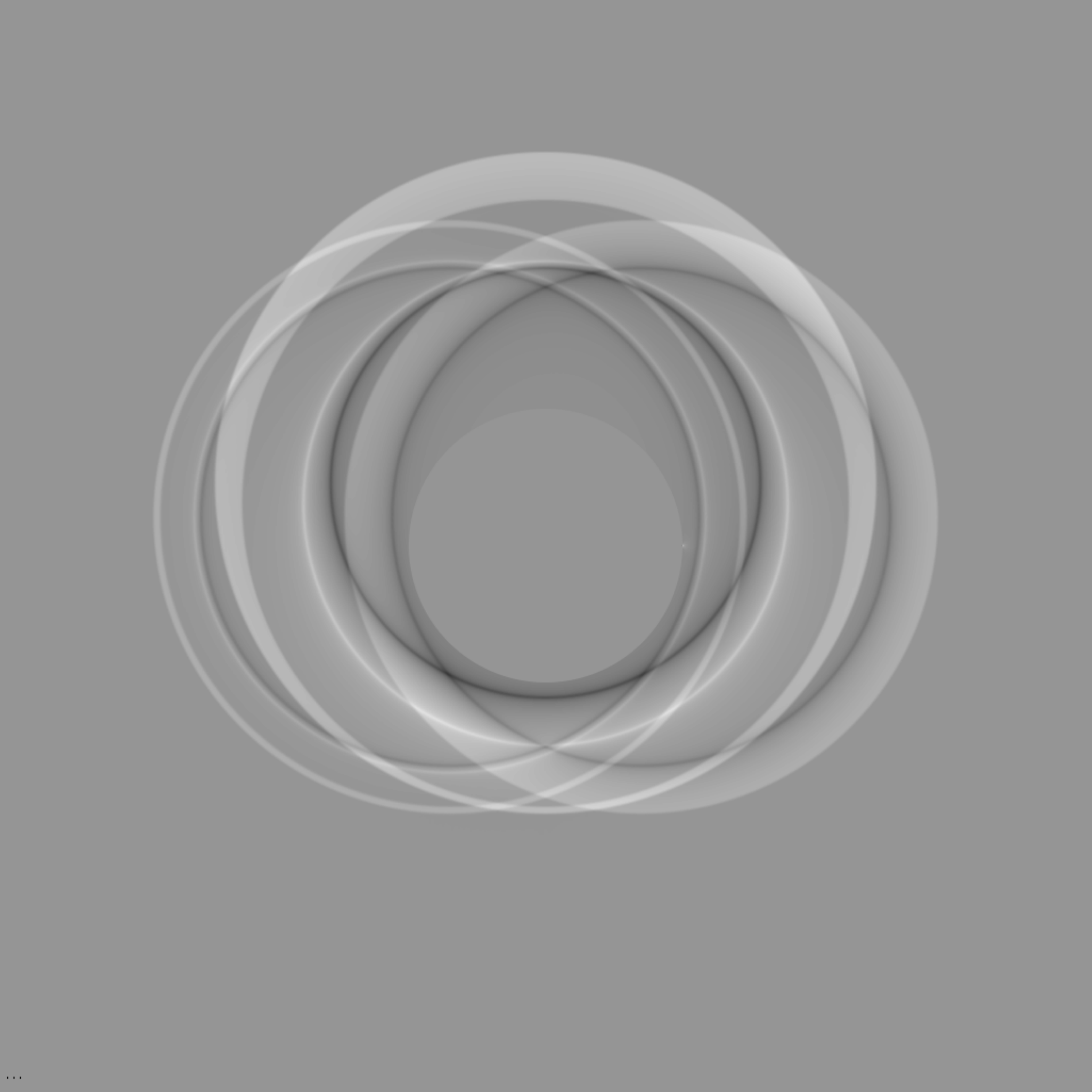} \ \includegraphics[scale = 0.65]{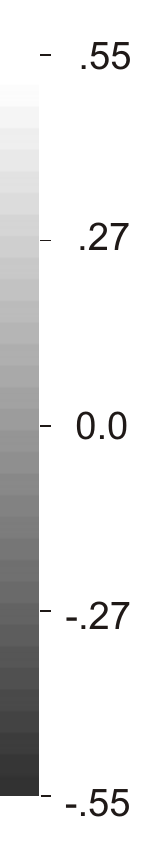}}
\subfigure[Partial data solution $w(t,x)$ at $t=2$ ]{
\includegraphics[scale = 0.15]{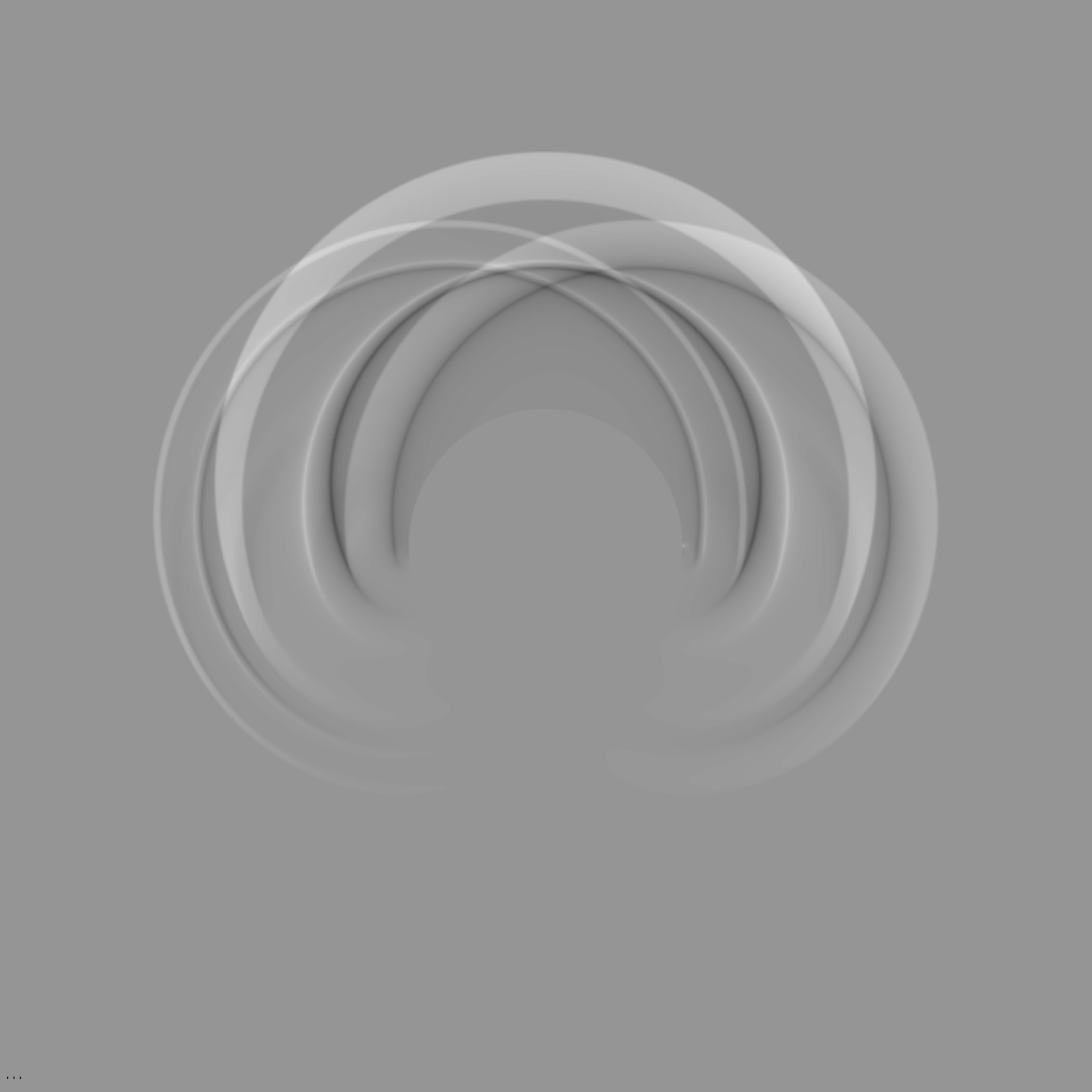} \ \includegraphics[scale = 0.65]{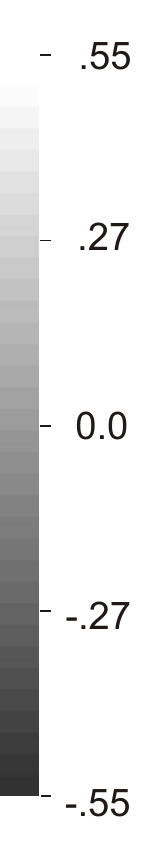}}
\subfigure[Difference $w(t,x)-u(t,x)$ at $t=2$ ]{
\includegraphics[scale = 0.15]{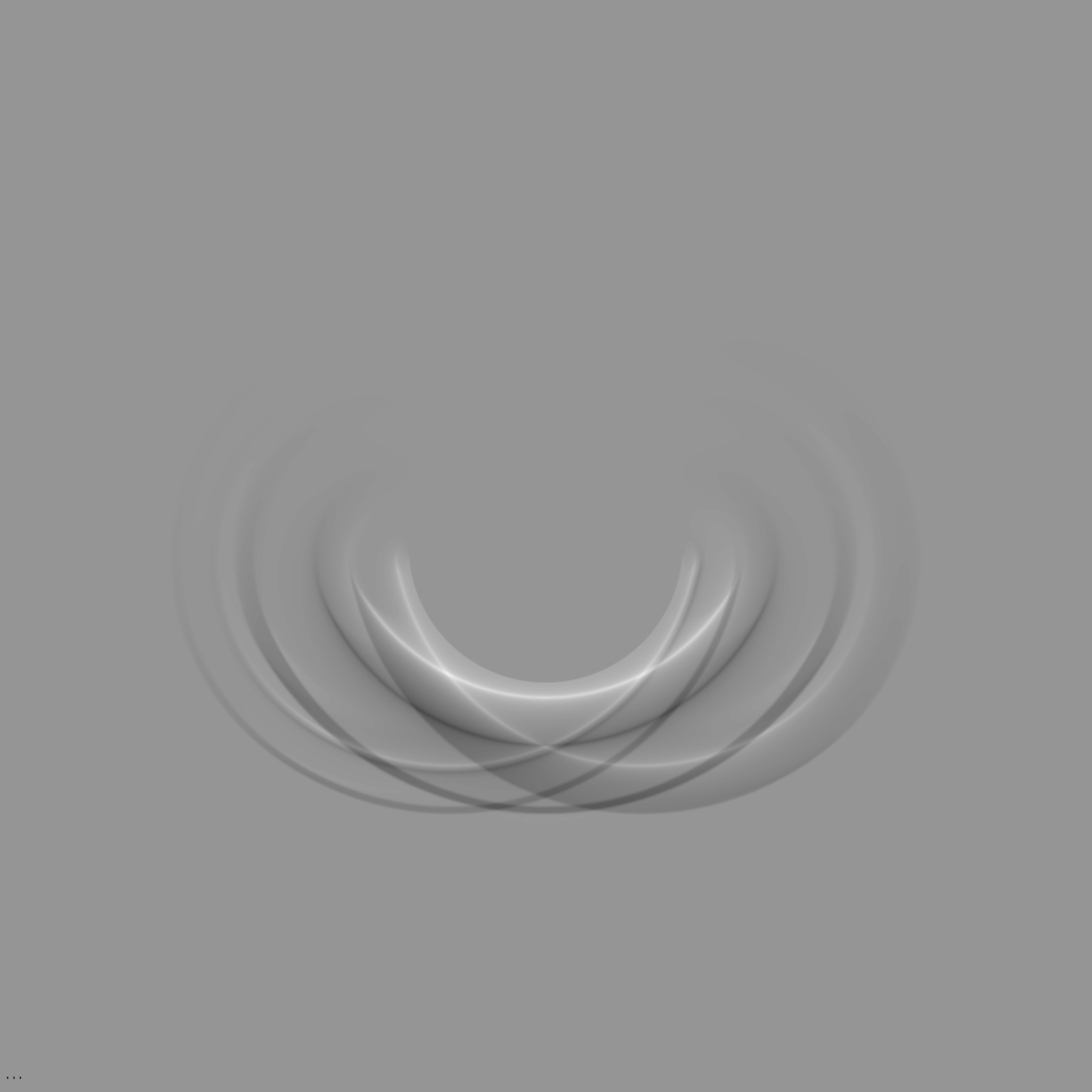} \ \includegraphics[scale = 0.65]{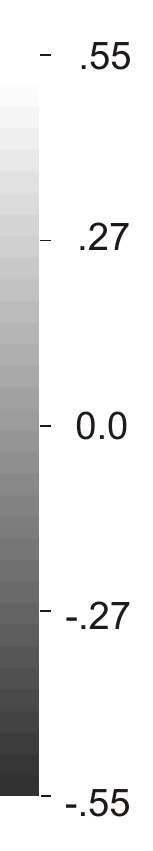}}
\end{center}
\caption{Comparing solution $u(t,x)$ of the direct problem and solution
$w(t,x)$ of the exterior problem with reduced data}%
\label{F:wave}%
\end{figure}

\section{Numerical simulations}

In this section performance of our algorithm is demonstrated in numerical
simulations. The model acquisition scheme can be seen in Figure
\ref{F:circgeom}. Measurements where modeled by using 256 samples over time
interval [0,2] (slightly more were used to produce a smooth cut-off after
$T=2).$ There were 1024 detector locations uniformly distributed over the
angle interval between $0$ and $360^{\circ}.$ These were used to model full
measurements. When we modeled the reduced measurements, all the values
corresponding to angle interval $[190^{\circ},350^{\circ}]$ were set to zero.
We implemented the algorithm described in the previous section in MATLAB. The
time required to compute $\frac{\partial}{\partial p}F(\omega,p)$ on a
$256\times1024$ grid using one thread of the Intel Xeon processor E5620
running at 2.4 GHz was about 1 sec. About 85\% of that time was used to
compute values of the Hankel functions utilized by the algorithm. This suggests
that, if needed, computations can be significantly sped up by either pre-computing
the Hankel functions, or by implementing the algorithm in C or Fortran
that compute these functions significantly faster (cf. the computing times reported in~\cite{Kun-cyl}).
\begin{figure}[t]
\begin{center}
\subfigure[Projections  reconstructed from full data $g(\omega,p)$.]{
\includegraphics[height = 1.4in , width = 2.64in]{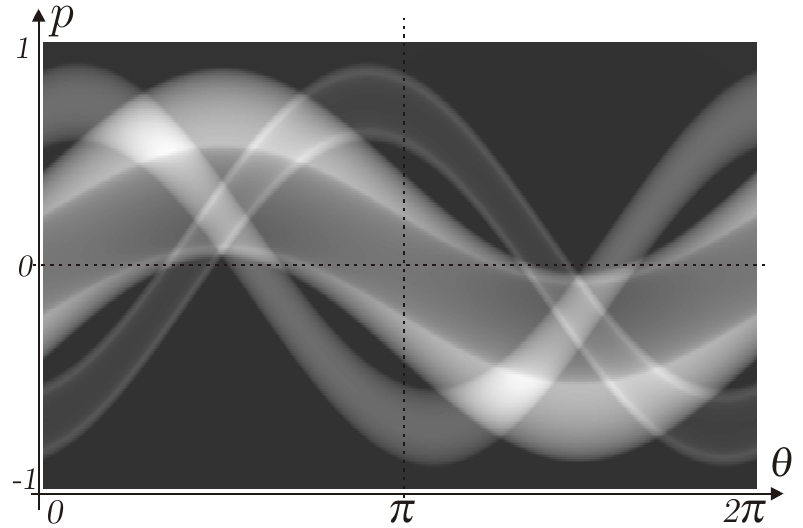} \phantom{aa} \includegraphics[scale = 0.52]{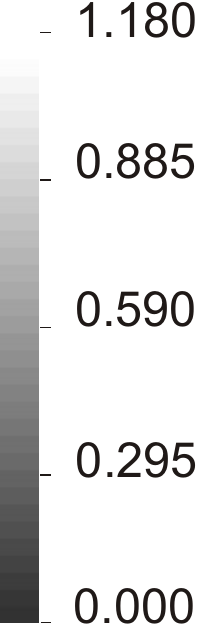}}
\newline
%\subfigure[Approximate projections $\hat{F}(\omega(\theta),p)=(\mathcal{R}w(2,x))(\omega(\theta),p)$]{
\subfigure[Approximate projections $(\mathcal{R}w(2,x))(\omega(\theta),p)$]{
\includegraphics[height = 1.4in , width = 2.64in]{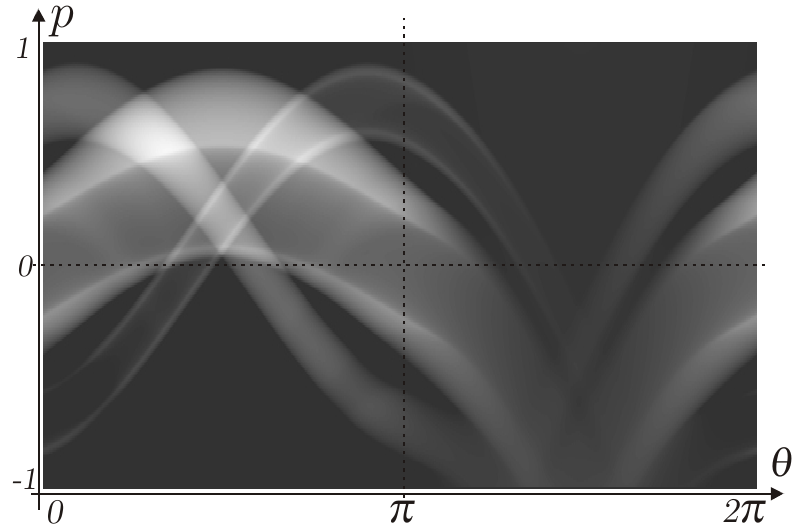} \phantom{aa} \includegraphics[scale = 0.52]{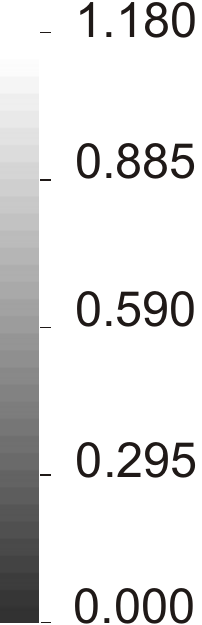}}
\newline
%\subfigure[Using symmetry to obtain the final approximation $G(\omega(\theta),p)$]{
\subfigure[Final approximation $G(\omega(\theta),p)$]{
\includegraphics[height = 1.4in , width = 2.64in]{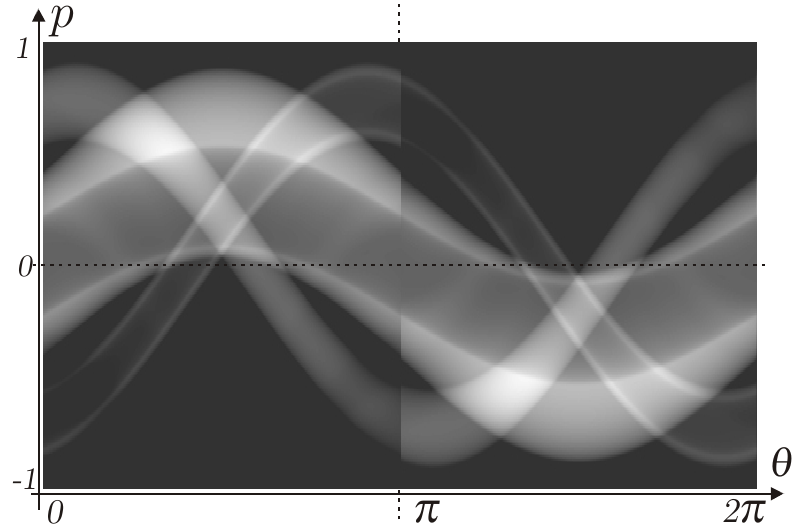} \phantom{aa} \includegraphics[scale = 0.52]{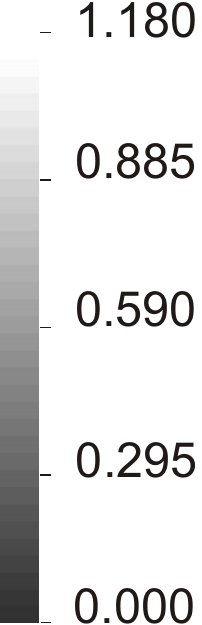}}
\newline
\end{center}
\caption{Comparing exact projections and the approximations obtained by the
present technique. Here $\omega(\theta) = (\cos(\theta), \sin(\theta) )$}%
\label{F:projections}%
\end{figure}

Figure \ref{F:wave}(b) presents the solution $u(t,x)$ at $t=2$ of the problem
(\ref{E:fulldataexterior}) with full data. Recall that theoretically, in
$\mathbb{R}^{d}\backslash B(0,1)$ this function coincides with the solution of
the forward problem (\ref{E:originalwave}), (\ref{E:originalBC}) with the
initial condition $f(x)$ shown in part (a) of the figure. Solution of the
initial/boundary problem (\ref{E:Eller2}) with reduced data is demonstrated in
Figure \ref{F:wave}(c). The difference $w(2,x)-u(2,x)$ (i.e. the error induced
by the loss of data) is plotted in the part (d) of the figure. One can see
that the singularities in this difference are all located within the lower
half space. Moreover, the error within the upper half-space is smooth and
quite small.

Figure \ref{F:projections}(a) presents the Radon projections $F(\omega,p)$
computed using our algorithm from the full data. Our algorithm is
theoretically exact in this case, and $F(\omega,p)$ is indistinguishable from
the exact Radon projections. Figure \ref{F:projections}(b) demonstrates
approximate projections $\tilde{F}(\omega(\alpha),p)$ computed from the
reduced data. One can see that on the interval $\alpha\in\lbrack0,\pi]$ the
approximate projections are quite close to $F(\omega,p).$ Part (c) of this
Figure shows projections $G(\omega(\alpha),p)$ computed from $\tilde{F}%
(\omega(\alpha),p)$ using symmetry relation (\ref{E:G}). One can notice a
small jump of values occurring at $\alpha=\pi.$ However, according to
Proposition \ref{T:final} this singularity will not propagate into the
reconstructed image. \begin{figure}[t]
\begin{center}
\subfigure[Reconstruction from full data]{
\includegraphics[scale = 0.2]{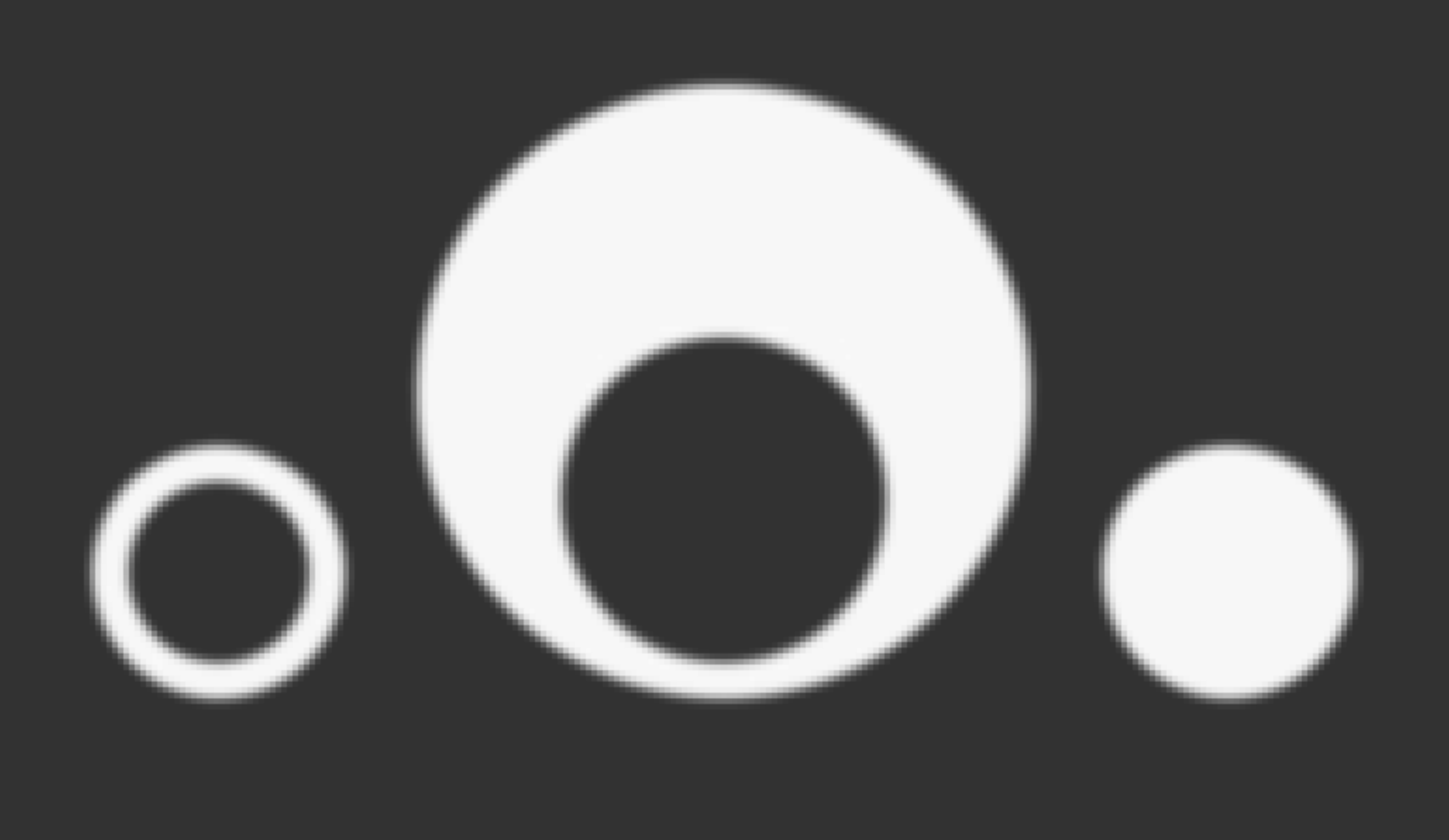} \ \includegraphics[scale = 0.36]{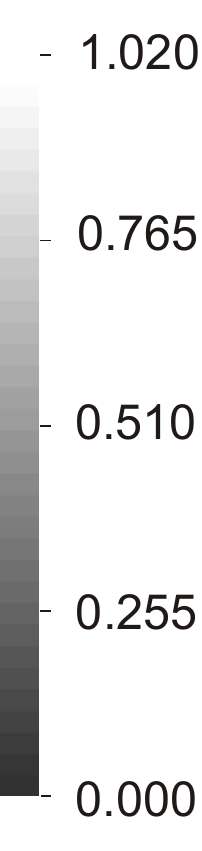}}
\subfigure[Na\"{i}ve reconstruction from reduced data]{
\includegraphics[scale = 0.2]{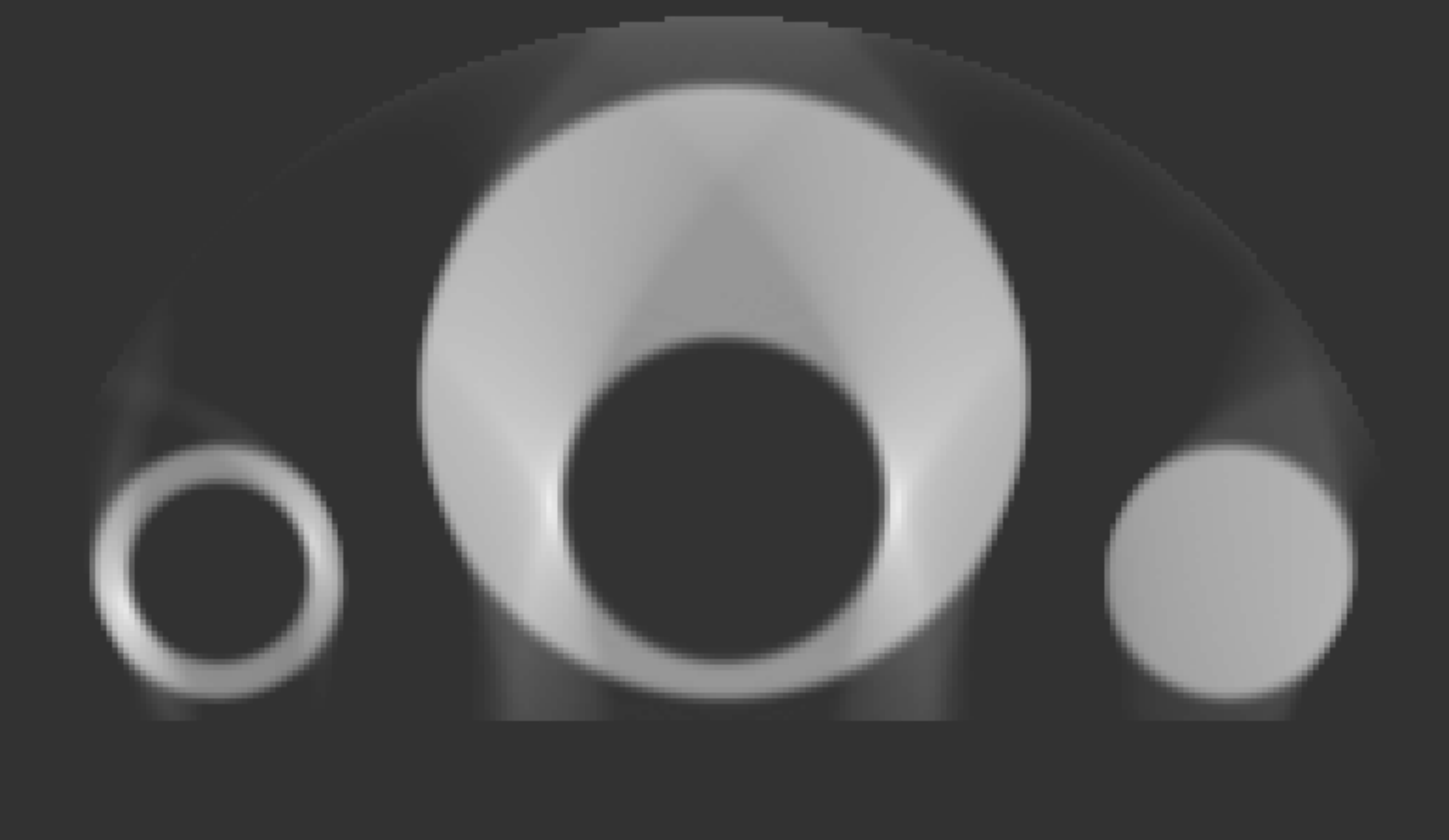} \ \includegraphics[scale = 0.36]{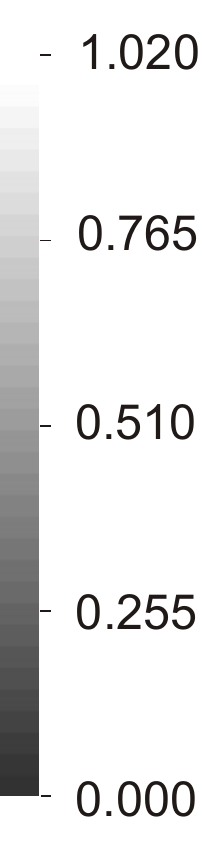}}
\subfigure[Present method using reduced data]{
\includegraphics[scale = 0.2]{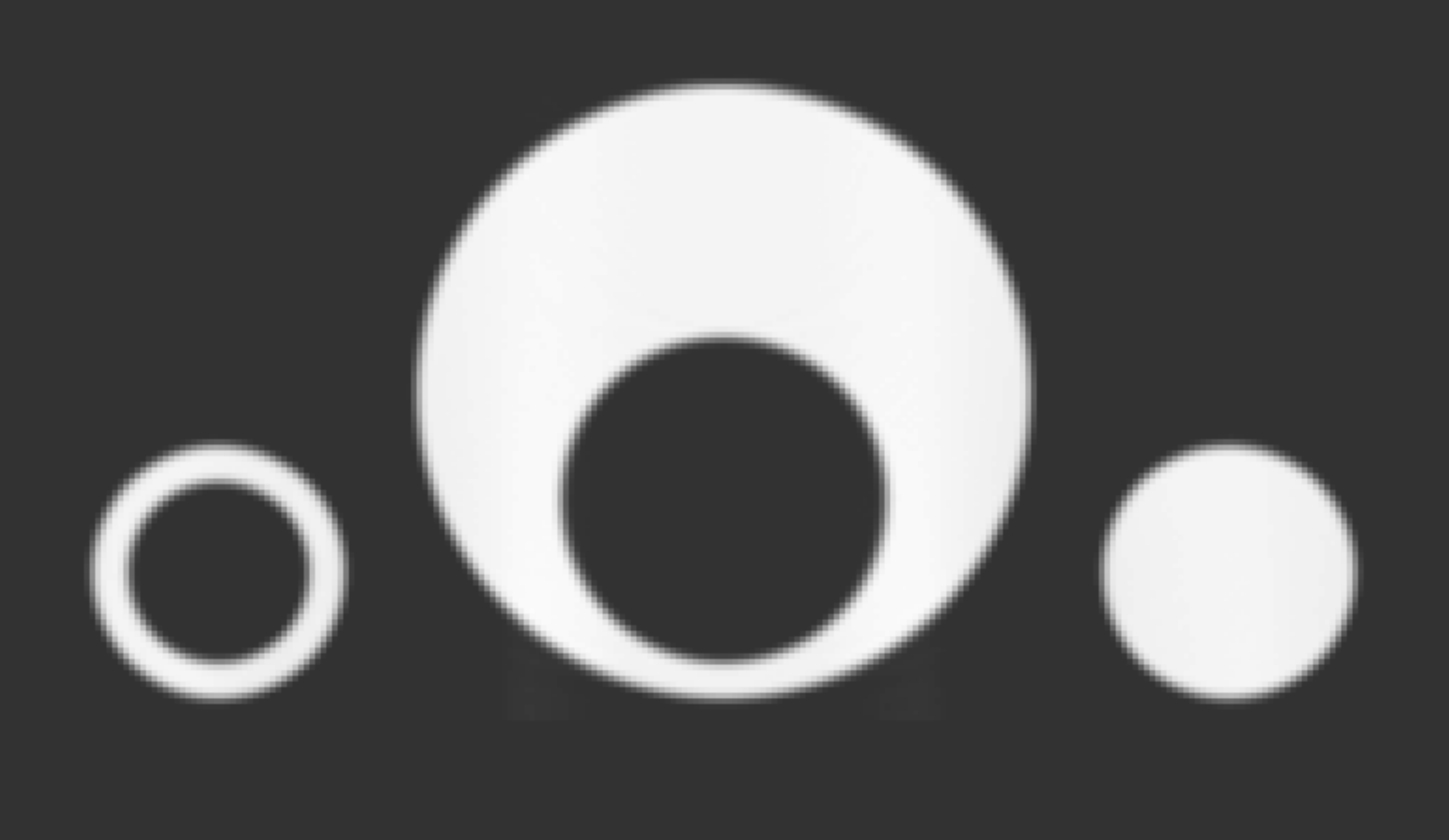} \ \includegraphics[scale = 0.36]{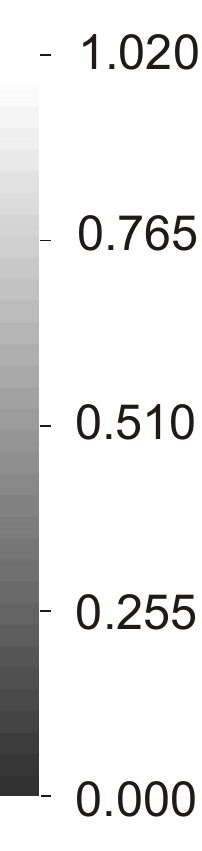}}
\subfigure[Diffirence between subfigures (a) and (c)]{
\includegraphics[scale = 0.2]{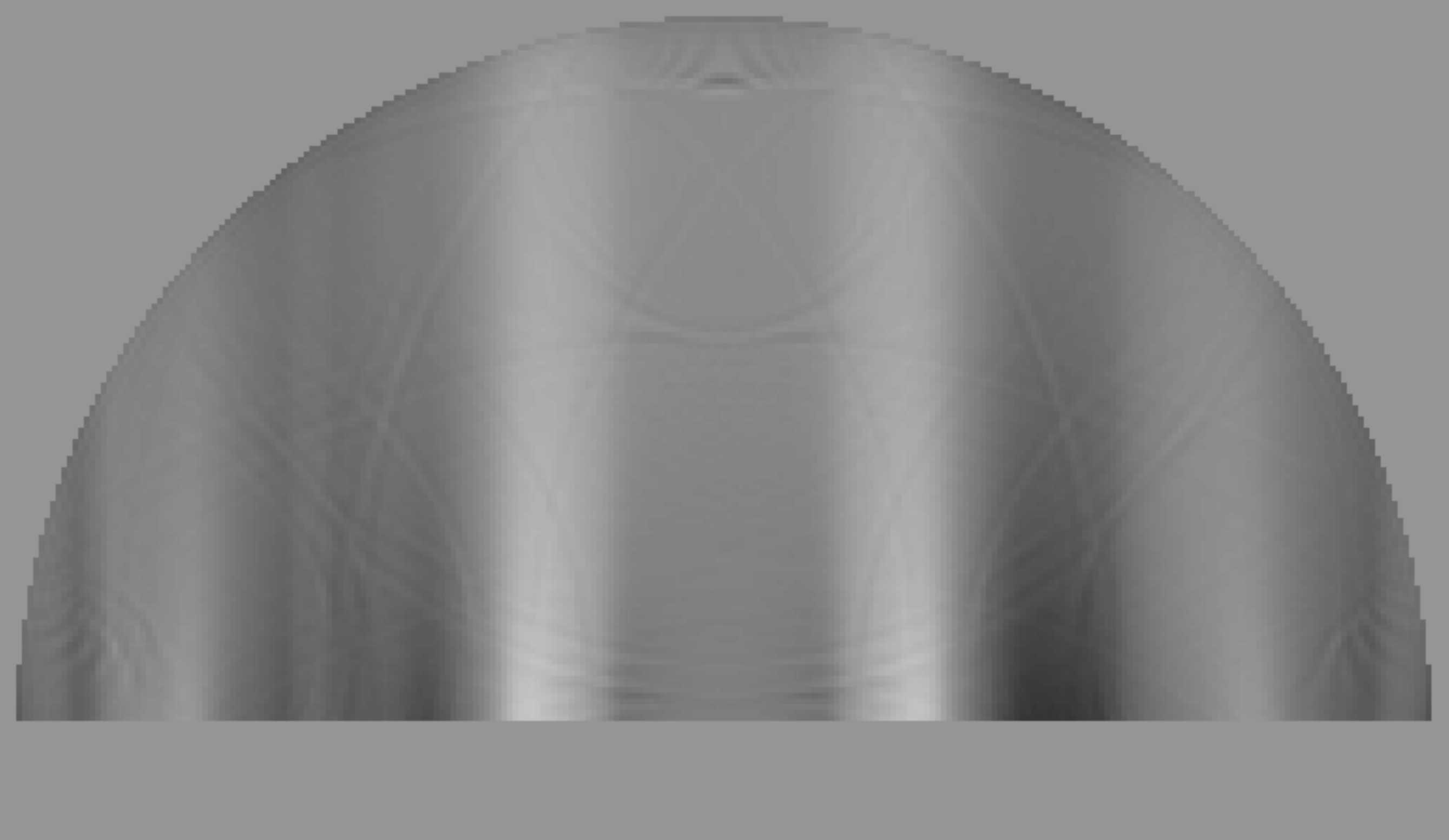} \ \includegraphics[scale = 0.36]{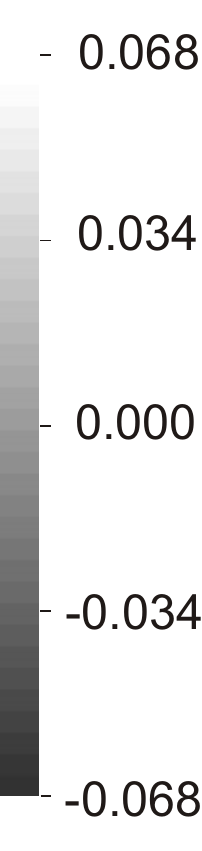}}
\end{center}
\caption{Comparing the results of the present technique with the full-data
reconstruction (subfigure (a)) and with a na\"{i}ve reconstruction not taking into
account the loss of data}%
\label{F:reconstructions}%
\end{figure}

\begin{figure}[t]
\begin{center}
\subfigure[Horizontal cross sections ]{
\includegraphics[width=3.0in,height=1.1in]{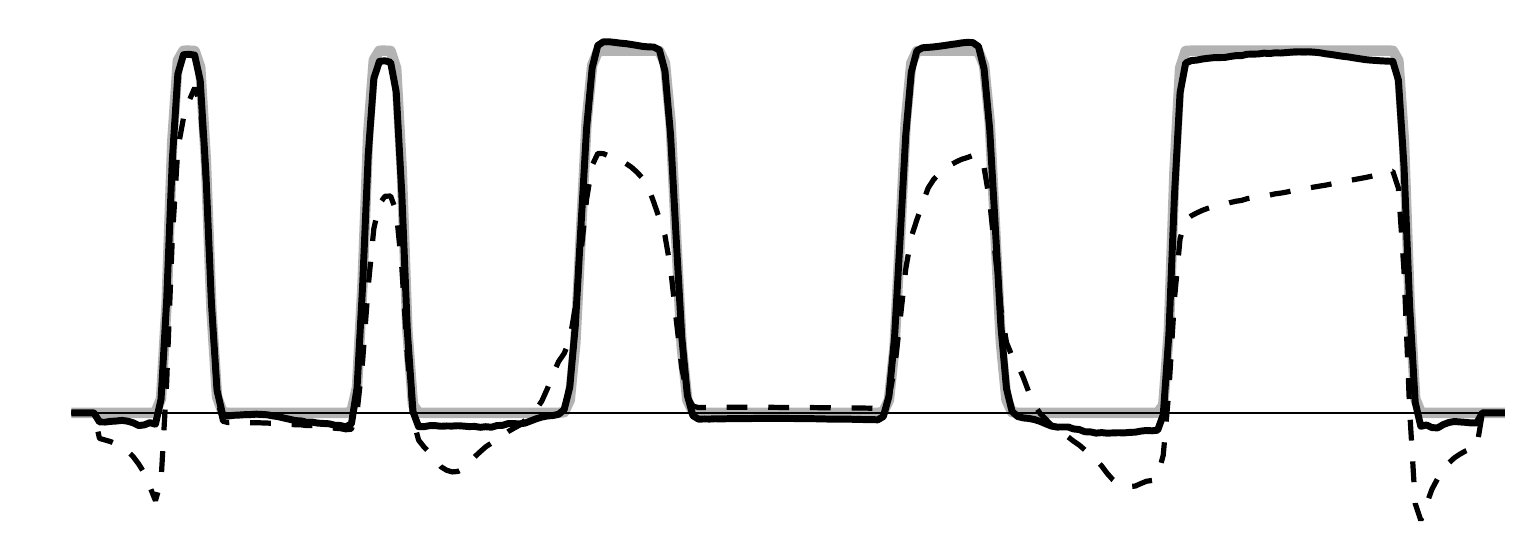}}
\subfigure[Vertical central cross sections ]{
\includegraphics[width=1.8in,height=1.08in]{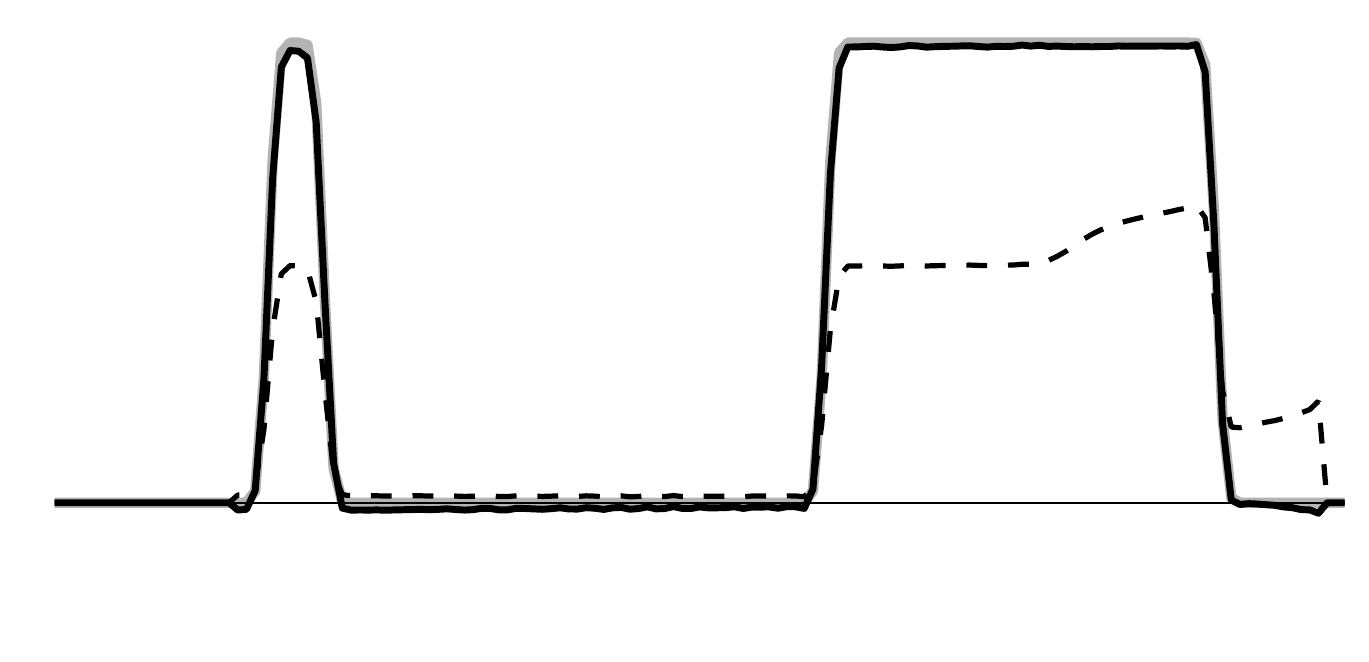}}
\end{center}
\caption{Cross sections of the phantom $f(x)$ and its reconstructions. The
gray line is the phantom, the black line represents the result obtained by the
present technique, and the dashed line shows the na\"{i}ve reconstruction}%
\label{F:profiles}%
\end{figure}

Figures \ref{F:reconstructions} and \ref{F:profiles} demonstrate the results
of applying the inverse Radon transform to the correct projections
$F(\omega,p)$ and to our approximation $G(\omega(\alpha),p).$ In detail,
theoretically exact reconstruction $\mathcal{R}^{-1}F$ is shown in Figure
\ref{F:reconstructions}(a). Part (b) shows the image obtained by a na\"{i}ve
reconstruction from reduced data; it was computed using the fast 2D algorithm
from \cite{Kun-cyl}. Part (c)\ of Figure \ref{F:reconstructions} presents
microlocally accurate reconstruction $\mathcal{R}^{-1}G.$ Visually it differs
little from $\mathcal{R}^{-1}F$ exhibited in part (a). The error is about 6\%
in the relative $L_{\infty}$ norm and about 3\% in the relative $L_{2}$
norm. The difference $\mathcal{R}^{-1}F-\mathcal{R}^{-1}G$ can be seen in
Figure \ref{F:reconstructions}(d), plotted using a much finer color scale. It
appears to be smooth, in accordance with the analysis presented in the
previous sections.

We have also plotted in Figure \ref{F:profiles} the graphs of the horizontal
(along the line $x_{2}=0.2$) and central vertical cross-sections the phantom
$f(x)$ and images shown in Figures \ref{F:reconstructions} (b) and (c). One
can see that while the na\"{i}ve reconstruction is incorrect even qualitatively,
our technique yields a close fit to the ground truth $f(x).$

\textbf{Acknowledgements} The authors would like to thank L. Friedlander and E.T.
Quinto for helpful discussions and useful literature references. The
last two authors gratefully acknowledge partial support from the NSF, award
NSF/DMS 1814592.
M. Eller and L.Kunyansky are grateful to the University of L\"ubeck for
organizing the workshop ``Modeling, analysis, and approximation theory toward applications in tomography and inverse problems''  in June 2016 supported by the DFG (German Science Foundation), that initiated their collaboration.

\bibliographystyle{alpha}
\bibliography{leonid,article}

\end{document}